\documentclass[12pt,twoside]{amsart}
\usepackage{amssymb}
\usepackage{amscd}
\usepackage{enumerate}
\RequirePackage[dvipsnames,usenames]{xcolor}

\usepackage{hyperref}
\hypersetup{
bookmarks,
bookmarksdepth=3,
bookmarksopen,
bookmarksnumbered,
pdfstartview=FitH,
colorlinks,backref,hyperindex,
linkcolor=Sepia,
anchorcolor=BurntOrange,
citecolor=MidnightBlue,
citecolor=OliveGreen,
filecolor=BlueViolet,
menucolor=Yellow,
urlcolor=OliveGreen
}
\usepackage{xypic}
\xyoption{curve}
\usepackage{amssymb,amsthm,amsmath}
\usepackage{mathtools}
\usepackage{mathrsfs}
\newcommand{\cO}{\mathcal{O}}

\title[Rational connectedness of globally $F$-regular threefolds]
{On rational connectedness of globally $F$-regular threefolds}
\author[Gongyo]{Yoshinori Gongyo}
\author[Li]{Zhiyuan Li}
\author[Patakfalvi]{Zsolt Patakfalvi}
\author[Schwede]{Karl Schwede}
\author[Tanaka]{Hiromu Tanaka}
\author[Zong]{Hong Runhong Zong}
\date{\today, version 0.03}
\address{Graduate School of Mathematical Sciences, the University of Tokyo, 3-8-1 Komaba, Meguro-ku, Tokyo 153-8914, Japan.}
\email{gongyo@ms.u-tokyo.ac.jp}
\address{Department of Mathematics, Imperial College London, 180 Queen's Gate, London SW7 2AZ, UK.}
 \email{y.gongyo@imperial.ac.uk}

\address{Room 382-B, Department of Mathematics, Building 380, Stanford, CA 94305. }
\email{zli2@stanford.edu}

\address{Princeton University, Department of Mathematics, Phone: (206) 436-9676, Fine Hall, Washington Road, Princeton, NJ 08544-1000.}
\email{pzs@princeton.edu}

\address{ Department of Mathematics, University of Utah, Salt Lake City, UT, USA, 84112.}
\email{schwede@math.utah.edu}

\address{Department of Mathematics, Faculty of Science, Kyoto University, Kyoto 606-8502 Japan.}
\email{tanakahi@math.kyoto-u.ac.jp}

\address{Room A5 Fine Hall, Washington Road, Department of Mathematics, Princeton University, Princeton NJ 08544 US.}
\email{rzong@math.princeton.edu}
 \thanks{The first author was partially supported by the Grand-in-Aid for Research Activity Start-Up $\sharp$24840009 from JSPS and Research expense from the JRF fund.  The fourth author was partially supported by the NSF grant DMS \#1064485, NSF FRG grant DMS \#1265261 NSF CAREER grant DMS \#1252860 and by a Sloan Fellowship.
The fifth author is partially supported by
the Research Fellowships of the JSPS for Young Scientists (24-1937). }
\newcommand{\Spec}[0]{{\operatorname{Spec}}}
 \newcommand{\Hom}[0]{{\operatorname{Hom}}}

\newcommand{\Supp}[0]{{\operatorname{Supp}}}

 \newcommand{\Ex}[0]{{\operatorname{Ex}}}
\newtheorem{thm}{Theorem}[section]
\newtheorem{lem}[thm]{Lemma}
 \newtheorem{cor}[thm]{Corollary}
\newtheorem{prop}[thm]{Proposition}

\newtheorem{cl}[thm]{Claim}

\theoremstyle{definition}

\newtheorem{dfn}[thm]{Definition}

\newtheorem{rem}[thm]{Remark}
\newtheorem*{ack}{Acknowledgments}
\newtheorem{nota}[thm]{Notation}
\newtheorem{step}{Step}
\newtheorem{case}{Case}
\newtheorem{nasi}[thm]{}

\newtheorem{claim}[equation]{Claim}
 \newcommand{\Q}{$\mathbb{Q}$}

\newcommand{\sHom}{\mathscr{H}om}
\newcommand{\sR}{\mathscr{R}}

\newcommand{\blank}{\ \ }

\begin{document}

\maketitle

\begin{abstract}
In this paper,
we show that
projective globally $F$-regular threefolds,
defined over an algebraically closed field of characteristic $p\geq 11$,
are rationally chain connected.
%It is conjectured that globally $F$-regular variety are separably rationally connected. This conjecture follows from some standard conjectures in positive characteristic higher dimensional geometry. Then we show that varieties of globally $F$-regular type are rationally connected.

%In dimension $3$, we show the rationally chain connectedness of globally $F$-regular $\mathbb{Q}$-factorial Gorenstein terminal projective threefolds over algebraically closed field of characteristic $p \geq 7$. Furthermore,  if $p \geq 11$, then every such threefold is either birational to a $\mathbb{Q}$-Fano threefold of Picard number 1 or   separably rationally connected.
\end{abstract}

\tableofcontents

\setcounter{section}{-1}

%%%%%
\section{Introduction}
\numberwithin{equation}{thm}

In the 90's, Campana and Koll\'ar--Miyaoka--Mori showed that
smooth Fano varieties are rationally connected in characteristic zero (\cite{KMM}, \cite{campana}).
Recently Zhang and Hacon--$\mathrm{M^{c}}$Kernan generalized this to log Fano varieties in characteristic zero (\cite{Zhang}, \cite{HM}).
Then, it is natural to ask the following question:
are smooth Fano varieties in positive characteristic rationally connected?
This is still an open problem, but it is known
that they are rationally chain connected (cf. \cite[Ch V, 2.14]{Kollar2}).
In this paper, we consider a related problem on this result.
The main theorem of this paper is as follows.

\begin{thm}[Section \ref{section-main}]%[Theorem~\ref{RCC-Freg}]
\label{0gfr-rcc}
Let $k$ be an algebraically closed field of characteristic $p\geq 11$.
Let $X$ be a projective globally $F$-regular threefold over $k$.
Then, $X$ is rationally chain connected.
\end{thm}

%\begin{rem}
%One can see from the proof (see Theorem \ref{dim1-section-2}) that actually over an algebraically closed field of characteristic $p\geq 11$  any projective globally $F$-regular threefold is either birational to a $\mathbb{Q}$-Fano threefold or separably rationally connected. But it is open in positive characteristic whether  $\mathbb{Q}$-Fano threefolds are separably rationally connected or not.
%\end{rem}
In some cases, we can show that $X$ is in fact separably rationally connected.

\begin{thm}[Theorem~\ref{dim1-section-2}]
\label{0gfc-src}
%\label{dim1-section-2}
Let $k$ be an algebraically closed field of characteristic $p>0$.
Let $f:X\to Y$ be a projective surjective morphism between normal varieties over $k$ with $f_*\mathcal{O}_X=\mathcal{O}_Y$.  Assume that the following conditions hold.
\begin{enumerate}
\item{$p\geq 11$.}
\item{$X$ is a terminal globally $F$-regular threefold.}
\item{$-K_X$ is $f$-ample.}
\item{$\dim Y=1$ or $\dim Y=2.$}
\end{enumerate}
Then $X$ is separably rationally connected (SRC).
\end{thm}

Globally $F$-regular varieties were introduced by K.~Smith \cite{Sm} who drew inspiration from both the theory of tight closure and the theory of Frobenius split varieties.
Global $F$-regularity is a global property of a projective variety over a  field of positive characteristic and it imposes strong conditions on the structure of the variety.

The assumption that $X$ is globally $F$-regular is analogous to, but more restrictive than
assuming that $X$ is log Fano (for some boundary divisor $\Delta$).
Indeed, there exists del Pezzo surfaces in characteristic $p\leq 5$ which are
not globally $F$-regular (cf. \cite{hara}).
On the other hand,
a log Fano variety $X$ in characteristic zero,
is globally $F$-regular type \cite{SS}.   This means that
almost all the modulo $p$ reductions $X_{p}$ of $X$ are globally $F$-regular.

\begin{nasi}[The strategy for the main theorem]  The methods of  \cite{Zhang} and \cite{HM} do not work in positive characteristic since their proofs depend on semi-positivity theorems and the extension theorem.  The proofs of these theorems rely on the Kodaira vanishing theorem or on Hodge theory.  As is well-known, both the Kodaira vanishing theorem and semi-positivity can fail in positive characteristic \cite{Raynaud} \cite[3.2]{Moret_Bailly_Familles_de_courbes_et_de_varietes_abeliennes_sur_P_1_II_exemples}.
Let us overview the proof of Theorem~\ref{0gfr-rcc}.
Let $X$ be a projective globally $F$-regular threefold.
First, we use the minimal model program for threefolds in characteristic $p > 0$, as is recently developed in
\cite{CTX} and \cite{HX}.
We replace $X$ by the result of this minimal model program which, since a globally $F$-regular variety is log Fano, is a Mori fiber space $f : X \to Y$.
Note that one must be careful here because the property of rational chain connectedness is not stable under birational model changes,
%The reason why any consideration is necessary here is that the property of rational chain connectedness is not stable under birational model changes,
e.g. the cone over elliptic curve is rationally chain connected, but
its blowup of the vertex is a $\mathbb P^1$-bundle over an elliptic curve, which is not rationally chain connected.

At this point, we have a Mori fiber space structure $f:X\to Y$.
For simplicity, we explain the proof in the case where $\dim Y>0$ (cf. Proposition~\ref{rcc-pic-one}).
It is enough to show that $Y$ and general fibers are rationally chain connected,
and that $f$ has a section.
It is easy to show the rational connectedness of $Y$ (cf. Lemma~\ref{glFness for MMP}).
The rational connectedness of general fibers holds by the following theorem.

\begin{thm}[Theorem~\ref{general fiber}]\label{0general-fiber}
%Let $k$ be a $F$-finite base-field of characteristic $p>0$.
Let $f : X \to Y$ be a proper morphism from a normal scheme to an $F$-finite integral scheme of characteristic $p > 0$, such that  $f_* \mathcal{O}_X \cong \mathcal{O}_Y$. Further let $\Delta$ be an effective $\mathbb{Q}$-divisor on $X$, such that $(X, \Delta)$ is globally $F$-regular. Then there is a non-empty open set $U \subseteq Y$, such that for every geometric point $y \in U$, $(X_y, \Delta_y)$ is globally $F$-regular.
\end{thm}

It remains to show that $f$ has a section.
However, we only establish somewhat weaker results which are sufficient for our purposes
cf. Proposition~\ref{dim1-section} and Proposition~\ref{p1-fib-lemma}.  These follow from
\cite{mh} and \cite{dj-st} (a positive characteristic analog of \cite{ghs}) respectively.
Since \cite{mh} needs the assumption that $p\geq 11$, so do we.
%\cite{dj-st} is a positive characteristic analogue of the result by Graber--Harris--Starr \cite{ghs}.
\end{nasi}

We show Theorem~\ref{0gfc-src} by employing some known results,
Theorem~\ref{0general-fiber} and Theorem~\ref{0src} below.

\begin{thm}[Proposition~\ref{dim1-src}]\label{0src}
Let $k$ be an algebraically closed field of positive characteristic.
Let $f:X \to Y$ be a surjective morphism such that  $X$ and $Y$ are smooth projective varieties over $k$ such that $Y$ is SRC and also that the geometric generic fiber of $f$ is smooth, irreducible and SRC.  Then $X$ is SRC.
%In particular, when $p \geq 11$, any globally F-regular threefold is either birational to a Q-Fano threefold or separably rationally connected.
\end{thm}

\medskip

We now turn back to characteristic zero.
As was already mentioned globally $F$-regular varieties are closely related to log Fano varieties.  A log Fano variety $X$ in characteristic zero
is globally $F$-regular type \cite{SS}.
The inverse assertion is an open problem.
The following theorem provides positive evidence towards this conjecture.

\begin{thm}[Theorem~\ref{rat-con}]\label{0gfrtype-rc}
Let $k$ be an algebraically closed field of characteristic zero.
Let $X$ be a $\mathbb{Q}$-Gorenstein normal projective variety over $k$
which is of globally $F$-regular type.
Then $X$ is rationally connected.
\end{thm}

 \begin{ack} This research started at the workshop \emph{The minimal model program in characteristic p} held at the American Institute of Mathematics (AIM) in May of 2013. The authors would like to thank AIM, especially Professors  James $\mathrm{M^{c}}$Kernan and Chenyang Xu, who were the  organizers of the workshop, for their hospitality. The authors would like to thank Professors Paolo Cascini, Brendan Hassett, Yujiro Kawamata, Amanda Knecht, J\'anos Koll\'ar, Shunsuke Takagi, Zhiyu Tian, and Chenyang Xu for discussions. In particular, the 1st step of the proof of Theorem \ref{0gfr-rcc} depends on Professor Koll\'ar's advice and Proposition \ref{dim1-src} came out of discussions that Zsolt Patakfalvi and Zhiyuan Li had with Professor Hassett.  We would also like to thank the referee for many useful comments and thoughtful suggestions.
 \end{ack}

\section{Preliminaries}\label{section-preliminaries}

%\section{Preliminaries on rationally connected and globally $F$-regular varieties}\label{section-preliminaries}
We start with the definition of globally $F$-split and globally $F$-regular varieties:

\begin{dfn}
Let $X$ be a variety over an $F$-finite field $k$.  Then $X$ is \emph{(globally) $F$-split} if the map $\cO_X \to F_* \cO_X$ splits as a map of $\cO_X$-modules.
\end{dfn}

We need the definition of globally $F$-regular varieties for pairs.

\begin{dfn}
%\footnote{ \color{blue} Zsolt: Gongyo told me to add the definition of globally $F$-split to here, but I don't know what generality is needed. Is normal enough for example?}
Let $X$ be a normal variety over an $F$-finite field $k$ and
let $\Delta$ be an effective \Q-divisor on $X$.
We say $(X, \Delta)$ is \emph{globally $F$-regular} if
for every effective \Q-divisor $E$, there exists $e\in\mathbb Z_{>0}$ such that
the composition homomorphism
$$\mathcal O_X\overset{F^e}\to F_*^e\mathcal O_X\hookrightarrow F^e_* \mathcal O_X(\ulcorner(p^e-1)\Delta+E\urcorner)$$
 splits as an $\mathcal O_X$-module homomorphism where the latter map is the natural injection.

 If $X$ is quasi-projective (respectively G1\footnote{A variety is G1 if it is Gorenstein in codimension 1} and S2), but not a-priori normal, we say that $X$ is \emph{globally $F$-regular} if for every effective Cartier divisor $E \geq 0$ (respectively effective Weil divisorial sheaf $E$ which is Cartier in codimension 1), there exists an $e > 0$ such that the map
 \[
 \cO_X \overset{F^e}\to F_*^e\mathcal O_X\hookrightarrow F^e_* \mathcal O_X(E)
 \]
 splits as an $\cO_X$-module homomorphism.  In either case, it follows that $X$ is then normal by \cite[Corollary 5.11]{HHbs}.
\end{dfn}

\begin{rem}
The assumption that $X$ is G1 and S2 is a mild weakening of normality, see \cite{HartshorneGeneralizedDivisors} for a discussion.  In the next section we will frequently take a normal variety $X$ over a field $k$ and form the base change with a field extension $K \supseteq k$.  Unfortunately, the result $X_K$ may not be normal a priori.  However, because the G1 condition and the S2 condition are preserved by base field extension, $X_K$ is automatically G1 and S2.  In this paper the $X_K$ we form will later be proven to be globally $F$-regular and hence  a posteriori $X_K$ was normal all along.
\end{rem}

The following lemmas follow from the definition directly.

\begin{lem}\label{lem_1.2}
If $(X, \Delta)$ is a globally $F$-regular variety over an $F$-finite field,
then, for every $0\leq \Delta'\leq \Delta$, so is $(X, \Delta')$.
\end{lem}

\begin{lem}\label{glFness for MMP}Let $f: X \dashrightarrow X_1$ be a small birational map or a proper  morphism with $f_*\mathcal O_X=\mathcal O_{X_1}$ of normal varieties over an $F$-finite field of characteristic $p>0$.
If $X$ is globally $F$-regular (resp. globally $F$-split), then so is $X_1$.
\end{lem}

It turns out that globally $F$-regular varieties are always log Fano variety with respect to some boundary.

\begin{thm}[{\cite[Theorem 4.3]{SS}}]\label{ss1}
Let $X$ be a normal projective variety over an $F$-finite field of characteristic $p>0$.
If $X$ is globally $F$-regular $($resp. globally $F$-split$)$, then there exists an effective $\mathbb{Q}$-divisor $\Delta$ such that $(X, \Delta)$ is a globally $F$-regular pair (resp.  sharply globally  $F$-split pair ) and $-(K_X+\Delta)$ is ample $($resp.  $K_X+\Delta \sim_{\mathbb{Q}} 0$ $)$.  In particular,
$(X,\Delta)$ is Kawamata log terminal (resp. log canonical ).  \end{thm}

We now turn to some preliminary results about rationally connected varieties.
\begin{dfn} \cite[IV.3.2]{Kollar2}
\label{dfn.rationallyConnected}
Suppose that $X$ is a variety over a field $k$.
\begin{enumerate}
\item $X$ is \emph{rationally chain connected} if there is a proper family of connected algebraic curves $g : U \to Y$ such that the geometric fibers have only rational components and there is a cycle morphism ${u : U \to X}$ such that $u^{(2)} : U \times_Y U \to X \times_k X$ is dominant.
\item $X$ is \emph{rationally connected} if it is rationally chain connected as above but in addition the geometric fibers of $g$ are irreducible.
\item $X$ is \emph{separably rationally connected} if it is rationally connected and additionally $U = Y \times \mathbb{P}^1$ and $u^{(2)}$ is smooth at the generic point.
\end{enumerate}
\end{dfn}
%\begin{dfn}
%Let $X$ be a variety over a field $k$. It is \emph{separably rationally connected (SRC)} if there is a generically smooth family of $1$--cycles whose fibers geometrically rational %components:
%\[
%\begin{array}{ccccc}
% U  & \stackrel{u}{\longrightarrow} &  X\\
% g\downarrow &&  \\
%  B  &  &
%\end{array}
%\eqno{},
%\] over $k$ such that the double evaluation map:
%\[
%U\times_{B} U \stackrel{(u, u)}{\longrightarrow} X\times X
%\]
%is generically \'etale and dominates $X\times X$. If we drop the \'etale condition, $X$ is called \emph{rationally connected (RC)}. And if we do not require the generic smoothness %of $U\to B$, it is called rationally chain connected (RCC).
%\end{dfn}

We use the following base change property:

\begin{lem}[cf. {\cite[Ch IV, 3.2.5 Exercise]{Kollar2}}]\label{base_change_rcc}Let $k'/k$ be the field extension of algebraic closed fields and $X$ a proper variety over $k$. If $X'=X\times_{k}k'$ is rationally chain connected over $k'$, then so is $X$ over $k$

\end{lem}

\begin{proof}
By the assumption,
we can find a closed sub-scheme $B$ inside of $\mathrm{Chow}(X')$ such that
the double evaluation map
$U \times_B U \to X' \times X'$ of the corresponding family $U \to B$, as in Definition \ref{dfn.rationallyConnected}(1),
is surjective.
We can find $R$, an intermediate finitely generated $k$-algebra
$k \subset R \subset k'$ satisfying the subsequent condition: the map  $U \to B$ and
the double evaluation map are defined over $R$.
We then obtain $U_R\to B_R$. Further, we may assume that $B_R$ is irreducible.
%Also, by shrinking  ${\rm Spec}\,R$,
%we may assume that, for every scheme-theoretic point $\mathfrak p\in {\rm Spec} R$,
%the double evaluation map
%$$U_{\mathfrak p} \times_{B_{\mathfrak p}} U_{\mathfrak p}\to X'_{\mathfrak p} \times_{\mathfrak p} X'_{\mathfrak p}$$
%is surjective.

Then we claim the following:
\begin{cl}\label{claim21}
For every scheme-theoretic point $\mathfrak p \in \mathrm{Spec}\,R$,
the family $U_{\mathfrak p} \to B_{\mathfrak p} $ is a family of $1$-cycles with geometrically rational components.  Furthermore, after shrinking $\Spec\, R$ if necessary,
for every scheme-theoretic point $\mathfrak p\in {\rm Spec} R$
the double evaluation map
$$U_{\mathfrak p} \times_{B_{\mathfrak p}} U_{\mathfrak p}\to X'_{\mathfrak p} \times_{\mathfrak p} X'_{\mathfrak p}$$
is surjective.
\end{cl}
\begin{proof}[Proof of Claim \ref{claim21}]
By \cite[Ch II, Proposition 2.2]{Kollar2},
it is sufficient to show that,
for the generic point $\xi$ of $B_R$,
$U_{\xi}$ is a $1$-cycle with geometrically rational components.  Set $K$ to be the field of fractions of $R$.
Since $U = U_K \times_K k' \to B = B_K \times_K k'$ is a family of $1$-cycles with geometrically rational components,
so is $U_K\to B_K$.
Thus, it is enough to prove that
the morphism
$$B_K=B_R\times_R K\to B_R$$
is dominant, which holds by construction.  This proves the first part of the claim.

For the second part of the claim, we argue similarly.  We know that $U \times_B U \to X' \times X'$ is surjective and hence $U_K \times_{B_K} U_K \to X' \times X'$ is also surjective.  But then the same holds for all $\mathfrak{p}$ in an open dense set of  $\Spec\, R$.
\end{proof}
Thus $X$ is also RCC.
\end{proof}

\section{Global F-regularity and surjective morphisms}\label{section-general-fiber}

This section is devoted to the proof of the following main theorem:

\begin{thm}\label{general fiber}
%Let $k$ be a $F$-finite base-field of characteristic $p>0$.
Let $f : X \to Y$ be a proper morphism from a normal scheme to an $F$-finite integral scheme of characteristic $p > 0$, such that  $f_* \mathcal{O}_X \cong \mathcal{O}_Y$. Further let $\Delta$ be an effective $\mathbb{Q}$-divisor on $X$, such that $(X, \Delta)$ is globally $F$-regular. Then there is a non-empty open  set $U \subseteq Y$, such that for every perfect point $y \in U$, $(X_y, \Delta_y)$ is globally $F$-regular.
 \end{thm}

 \begin{rem}
 We make the following remarks on normality.  The hypothesis that $(X, \Delta)$ is globally $F$-regular implies that $X$ is globally $F$-regular and is hence normal.  It follows immediately that the general fiber is also normal.  The theorem then implies that the general geometric fibers are normal as well.
 \end{rem}

\begin{nota}
Throughout this section we only work with $F$-finite schemes and fields.  The reason for this is that the notion of $F$-splitting is much better behaved under this hypothesis.

Furthermore, and just in this section, we denote the $e$-th Frobenius pushforwards by using a $1/p^e$ exponent on the structure sheaf. I.e., instead of $F^e_* \mathcal{O}_X$ we write $\mathcal{O}_X^{1/p^e}$. The justification for this notation is that many isomorphisms such as $\mathcal{O}_X^{1/p^e} \otimes_{k^{1/p^e}} k^{1/p^{e+1}} \cong  \left(\mathcal{O}_X \otimes_{k} k^{1/p} \right)^{1/p^e}$ are much more apparent.
 \end{nota}

We first make the observation that $F$-split proper varieties remain $F$-split after base change of the base field.  We need one conceit that $H^0(X, \mathcal{O}_X) \supseteq k$ is a separable field extension.  Of course, if this is inseparable, then $H^0(X, \mathcal{O}_X \otimes_k k^{1/p}) \cong H^0(X, \mathcal{O}_X) \otimes_k k^{1/p}$ is clearly non-reduced so this additional hypothesis is certainly necessary.  Indeed, if $X$ is $F$-split and proper over $k$, with $H^0(X, \mathcal{O}_X) \supseteq k$ inseparable so that $k$ is non-perfect, then obviously $X \times_{k^p} k$ is not $F$-split.  Note that if we choose our field  $k = H^0(X, \mathcal{O}_X)$ then the separability hypothesis is automatic.

\begin{lem}
\label{lem.FSplitBaseChangeInsep}
Suppose that $k$ is an $F$-finite (but not necessarily perfect) field.  Further suppose that $X$ is a proper scheme over $k$ and $H^0(X, \mathcal{O}_X) =K$ is a separable field extension of $k$.  If $X$ is $F$-split, then every base change $X \times_k k^{1/p^e}$ is also $F$-split.  In particular $X$ is geometrically reduced \cite[Proposition 4.6.1.d]{EGA_IV_II}.
 \end{lem}
\begin{proof}
By induction on $e > 0$, it suffices to show that $X \times_k k^{1/p}$ is Frobenius split.
We have the following composition of Frobenius maps.
\[
\mathcal{O}_X \to \mathcal{O}_X \otimes_k k^{1/p} \to \mathcal{O}_X^{1/p} \otimes_{k^{1/p}} k^{1/p^2} \to \mathcal{O}_X^{1/p^2}.
 \]
We apply the functor
$\sHom_{\mathcal{O}_X}(\ \ , \mathcal{O}_X)$ to obtain:
\[
\begin{array}{cc}
& \mathcal{O}_X \leftarrow \sHom_{\mathcal{O}_X}(\mathcal{O}_X \otimes_k k^{1/p}, \mathcal{O}_X) \xleftarrow{\beta} \sHom_{\mathcal{O}_X}(\mathcal{O}_X^{1/p} \otimes_{k^{1/p}} k^{1/p^2}, \mathcal{O}_X)\\
& \leftarrow \sHom_{\mathcal{O}_X}(\mathcal{O}_X^{1/p^2}, \mathcal{O}_X)
\end{array}
 \]
Notice first that $\sHom_{\mathcal{O}_X}(\mathcal{O}_X \otimes_k k^{1/p}, \mathcal{O}_X) \cong \mathcal{O}_X \otimes_k \sHom_{k}(k^{1/p}, k) \cong \mathcal{O}_X \otimes_k k^{1/p}$, here $k$ is technically $f^{-1} k$ where $f : X \to \Spec k$ is the structural map.  Likewise
 \[
\begin{array}{rl}
      & \sHom_{\mathcal{O}_X}(\mathcal{O}_X^{1/p} \otimes_{k^{1/p}} k^{1/p^2}, \mathcal{O}_X) \\
\cong & \sHom_{\mathcal{O}_X \otimes_k k^{1/p}}(\mathcal{O}_X^{1/p} \otimes_{k^{1/p}} k^{1/p^2}, \sHom_{\mathcal{O}_X}(\mathcal{O}_X \otimes_k k^{1/p},\mathcal{O}_X))\\
 \cong & \sHom_{\mathcal{O}_X \otimes_k k^{1/p}}(\mathcal{O}_X^{1/p} \otimes_{k^{1/p}} k^{1/p^2}, \mathcal{O}_X \otimes_k k^{1/p})\\
\cong & \sHom_{\mathcal{O}_X \otimes_k k^{1/p}}((\mathcal{O}_X \otimes_{k} k^{1/p})^{1/p}, \mathcal{O}_X \otimes_k k^{1/p}).
\end{array}
 \]
Since $X$ is $F$-split, the canonical map $ \Hom_{\mathcal{O}_X}(\mathcal{O}_X^{1/p^2}, \mathcal{O}_X) \to H^0(X, \mathcal{O}_X)$ is surjective.  This implies that the map
\begin{equation*}
H^0(\beta) : \Hom_{\mathcal{O}_X \otimes_k k^{1/p}}((\mathcal{O}_X \otimes_{k} k^{1/p})^{1/p}, \mathcal{O}_X \otimes_k k^{1/p}) \to H^0(X, \mathcal{O}_X \otimes_k k^{1/p})
\end{equation*}
 is non-zero.  Now $H^0(X, \mathcal{O}_X) \cong K \supseteq k$ is a separable field extension and so $H^0(X, \mathcal{O}_X \otimes_k k^{1/p}) = H^0(X, \mathcal{O}_X) \otimes_k k^{1/p} \cong K \otimes_k k^{1/p}$ is a field. Further note that $\beta$ is a $\mathcal{O}_X \otimes_k k^{1/p}$-module homomorphism, hence $H^0(\beta)$ is a $H^0(X, \mathcal{O}_X \otimes_k k^{1/p})$-module homomorphism.  Thus $H^0(\beta)$ is a non-zero homomorphism over a field with the target being a one-dimensional vector space. Hence $H^0(\beta)$ is surjective.   This proves that $X \times_k k^{1/p}$ is also $F$-split.
 \end{proof}

As a quick corollary, we also obtain that globally $F$-split varieties are geometrically globally $F$-split.

\begin{cor}
\label{cor.BaseChangeOfFSplit}
Suppose that $X$ is a proper $F$-split variety over an $F$-finite field $k$ such that $H^0(X, \mathcal{O}_X) \supseteq k$ is a separable field extension.  Then $X$ is geometrically $F$-split.  In other words, for every field  extension $k \subseteq K$, for $K$ also $F$-finite, $X \times_k K$ is $F$-split.
 \end{cor}
\begin{proof}
Since $\mathcal{O}_X \otimes_k k^{1/p} \to \big( \mathcal{O}_X \otimes_k k^{1/p} \big)^{1/p}$ is split, so is $\mathcal{O}_X \otimes_k k^{1/p} \to \mathcal{O}_X^{1/p}$ since the latter factors the former.  Thus there is a map $\phi : \mathcal{O}_X^{1/p} \to \mathcal{O}_X \otimes_k k^{1/p}$ which sends $1$ to $1$.  It follows that $$(\phi \otimes_{k^{1/p}} K^{1/p}) : \mathcal{O}_X^{1/p} \otimes_{k^{1/p} } K^{1/p} \cong (\mathcal{O}_X \otimes_k K)^{1/p} \to \mathcal{O}_X \otimes_k K^{1/p}$$ also sends $1$ to $1$.
 But obviously since $K \to K^{1/p}$ is split, so is the map $\mathcal{O}_X \otimes_k K \to \mathcal{O}_X \otimes_k K^{1/p}$ and so there is an $\cO_X \otimes_k K$-linear map
 \[
 \beta : \mathcal{O}_X \otimes_k K^{1/p} \to \mathcal{O}_X \otimes_k K
 \]
  which also sends $1$ to $1$.  We compose and so obtain:
 \[
(\mathcal{O}_X \otimes_k K)^{1/p} \xrightarrow{(\phi \otimes_{k^{1/p}} K^{1/p})} \mathcal{O}_X \otimes_k K^{1/p} \xrightarrow{\beta}  \mathcal{O}_X \otimes_k K.
\]
This implies that $X \times_k K$ is $F$-split as desired.
%
%We first claim that if $k \subseteq K$ is inseparable, then $X \times_k K$ is $F$-split.  Certainly $K \subseteq k^{1/p^e}$ for some $e \gg 0$.  But then notice that since $K \hookrightarrow k^{1/p^e}$ splits, so does $\O_X \otimes_k K \hookrightarrow \O_X \otimes_k k^{1/p^e}$.  Thus since $X \times_k k^{1/p^e}$ is $F$-split, so is $X \times_k K$.
 %
%We may now assume that $k \subseteq K$ is separable.  Let $\phi : \O_X^{1/p} \to \O_X$ be a splitting.  We simply tensor $\phi$ by $K$ to obtain:  $\phi_K : \O_X^{1/p} \otimes_k K \to \O_X \otimes_k K$.  Observe that
 %\[\O_X^{1/p} \otimes_k K \cong \O_X^{1/p} \otimes_{k^{1/p}} k^{1/p} \otimes_k K \cong \O_X^{1/p} \otimes_{k^{1/p}} K^{1/p} \cong (\O_X \otimes_k K)^{1/p} \]
%where the last isomorphism follows from the fact that $K$ and $k^{1/p}$ are linearly disjoint over $k$.  But then we have a map $\phi_K : (\O_X \otimes_k K)^{1/p} \to \O_X \otimes_k K$ which obviously still sends $1$ to $1$.
 \end{proof}

The previous method also extends to globally $F$-regular varieties.

\begin{prop}
\label{prop:globally_F_regular_base_change}
Suppose that $k$ is an $F$-finite (but not necessarily perfect) field. Further suppose that $X$ is a proper normal variety over $k$, $\Delta$ an effective $\mathbb{Q}$-divisor on $X$ %that avoids the codimension one points that are not geometrically regular over $k$
and $H^0(X, \mathcal{O}_X) =K$ is a separable field extension of $k$.  If $(X,\Delta)$ is globally $F$-regular, then every base change $(X \times_k k^{1/p^e}, \Delta \times_k k^{1/p^e})$ is also globally  $F$-regular and in particular normal. %In particular $X$ is geometrically reduced \cite[Proposition 4.6.1.d]{EGA_IV_II}.
 \end{prop}

\begin{proof}
By \cite[Proposition 3.12]{SS}, and making $\Delta$ larger if necessary, we may assume that $(p^a -1)\Delta$ is a $\mathbb{Z}$-divisor where $a$ is a positive integer.
By Lemma \ref{lem.FSplitBaseChangeInsep},
$X$ is geometrically reduced over $k$, hence there is an effective divisor $D > 0$ on $X$ containing the support of $\Delta$ such that $X \setminus D$ is affine and smooth over $k$. Since $X$ is normal, it is geometrically G1 and S2, hence we can talk about divisors (or at worst Weil-divisorial sheaves) on base changes of $X$.
 We introduce the notation $\sR:=\mathcal{O}_X \otimes_{k} k^{1/p}$ and factor
\[
\xymatrix{X \ar@/_2pc/[rr]_{F_X^a} \ar[r]^-g & X' = \Spec_X \sR \ar[r]^-h & X}
\]
which is written on the structure sheaves as:
\[
\cO_X^{1/p^{a}} \leftarrow \cO_{X'} \cong \cO_X \otimes_k k^{1/p} = \sR \leftarrow \cO_X.
\]
\vskip 1pt
%\noindent so that $g$ is the relative $a$-iterated Frobenius.
 By induction on $e > 0$, the fact that $\Spec \sR$ is $X$ as a topological space, and \cite[Theorem 3.9]{SS}, it is enough to show that there is an integer  $e > 0$ with $a|e$ and so the following  natural inclusion is split:
 \begin{equation*}
\sR \to \Big(\sR \left( h^* (p^e -1) \Delta  +  h^* D  \right)\Big)^{1/p^e}.
\end{equation*}
Note that we may have to treat $h^* (p^e -1) \Delta  +  h^* D$ as an effective Weil divisorial sheaf which is Cartier in codimension one, since we do not know that $\sR$ is normal yet.
Regardless, we have the following composition of Frobenius maps.
\begin{equation}
\label{eq.GlobFRegMapToDualize}
\begin{array}{rcl}
 \mathcal{O}_X & \to & \sR \\
& \to & \Big(\sR  \left( h^*(p^e -1) \Delta   +  h^* D  \right) \Big)^{1/p^e}\\
& \to & \Big(\mathcal{O}_X\left(   (p^{e+a} -1) \Delta +  p^a D  \right)\Big)^{1/p^{e+a}}
\end{array}
\end{equation}
Here we used that $g^* h^* D = (F_X^a)^* D = p^aD$ and that
 \begin{equation*}
g^* h^* (p^e -1) \Delta  \cong (F_X^a)^* (p^e -1) \Delta  = (p^{e+a} -p^a) \Delta  \leq (p^{e+a} -1) \Delta
\end{equation*}
We now argue as before and apply the functor
$\sHom_{\mathcal{O}_X}(\blank, \mathcal{O}_X)$ to \eqref{eq.GlobFRegMapToDualize}
and obtain:
 \begin{multline*}
\mathcal{O}_X \leftarrow \sHom_{\mathcal{O}_X} \left(\sR, \mathcal{O}_X \right) \\
\xleftarrow{\beta} \sHom_{\mathcal{O}_X} \left( \Big(\sR  \left( h^*(p^e -1) \Delta + h^* D  \right)\Big)^{1/p^e}, \mathcal{O}_X \right)
\\ \leftarrow  \sHom_{\mathcal{O}_X} \left( \Big(\mathcal{O}_X\left(   (p^{e+a} -1) \Delta  +  p^a D  \right)\Big)^{1/p^{e+a}}, \mathcal{O}_X \right)
 \end{multline*}
Notice that
\[
\begin{array}{rl}
& \sHom_{\mathcal{O}_X}(\sR, \mathcal{O}_X) \\
= & \sHom_{\mathcal{O}_X}(\mathcal{O}_X \otimes_k k^{1/p}, \mathcal{O}_X) \\
\cong & \mathcal{O}_X \otimes \sHom_{k}(k^{1/p}, k) \\
\cong & \mathcal{O}_X \otimes_k k^{1/p}\\
 = & \sR,
\end{array}
\]
 here $k$ is technically $f^{-1} k$ where $f : X \to \Spec k$ is the structural map.  Likewise
 \[
\begin{array}{rl}
      & \sHom_{\mathcal{O}_X} \left( \Big(\sR  \left( h^* (p^e -1) \Delta +  h^* D  \right)\Big)^{1/p^e}, \mathcal{O}_X \right) \\
\cong & \sHom_{\sR} \left(\Big(\sR  \left( h^*(p^e -1) \Delta +  h^* D  \right)\Big)^{1/p^e}, \sHom_{\mathcal{O}_X}(\sR, \mathcal{O}_X) \right)\\
 \cong & \sHom_{\sR} \left( \Big(\sR  \left( h^* (p^e -1) \Delta +  h^* D  \right)\Big)^{1/p^e}, \sR \right)\\
\end{array}
\]
and so we identify $\beta$ with
\[
\sHom_{\sR} \left( \Big(\sR  \left( h^* (p^e -1) \Delta +  h^* D  \right)\Big)^{1/p^e}, \sR \right) \xrightarrow{\beta} \sR.
\]
Since $(X,\Delta)$ is globally $F$-regular, the canonical map
\begin{equation*}
 \Hom_{\mathcal{O}_X} \left( \left(\mathcal{O}_X \left(  (p^{e} -1) \Delta   +  D \right)\right)^{1/p^{e}}, \mathcal{O}_X \right) \to H^0(X, \mathcal{O}_X)
\end{equation*}
 is surjective for some $e > 0$ for which $a | e$ \cite[Proposition 3.8.a]{SS}.  But then since $(X, \Delta)$ itself is globally $F$-regular and so globally $F$-split, we have a splitting: $\left(\mathcal{O}_X \left(  (p^{a} -1) \Delta \right)\right)^{1/p^{a}} \to \mathcal{O}_X$.  Twisting by the divisor $(p^e - 1)\Delta + D$, reflexifying, and taking $p^e$th roots, we obtain a splitting:
  \[
 \left(\mathcal{O}_X \left(  (p^{e+a} - 1)\Delta+ p^a D\right)\right)^{1/p^{e+a}} \to \big(\mathcal{O}_X((p^e - 1)\Delta + D)\big)^{1/p^e}.
 \]
 We compose this with a splitting of $\mathcal{O}_X \to \left(\mathcal{O}_X \left(  (p^{e} -1) \Delta   +   D \right)\right)^{1/p^{e}}$ to obtain a splitting of $\mathcal{O}_X \to  \left(\mathcal{O}_X \left(  (p^{e+a} -1) \Delta   +  p^a D \right)\right)^{1/p^{e+a}}$.  Therefore,
\begin{equation}
\label{eq.nonZeroMap}
 \Hom_{\mathcal{O}_X} \left( \left(\mathcal{O}_X \left(  (p^{e+a} -1) \Delta   +  p^a D \right)\right)^{1/p^{e+a}}, \mathcal{O}_X \right) \to H^0(X, \mathcal{O}_X)
\end{equation}
 is surjective.
This implies that the map through which \eqref{eq.nonZeroMap} factors,
\begin{equation*}
H^0(\beta) : \Hom_{\sR}\left( \Big(\sR  \left( h^* (p^e -1) \Delta +  h^* D  \right)\Big)^{1/p^e}, \sR \right) \to H^0(X, \sR)
 \end{equation*}
is non-zero.

Just as before, $H^0(X, \mathcal{O}_X) \cong K \supseteq k$ is a separable field extension and so
\[
H^0(X, \sR) = H^0(X, \mathcal{O}_X \otimes_k k^{1/p}) = H^0(X, \mathcal{O}_X) \otimes_k k^{1/p} \cong K \otimes_k k^{1/p}
\]
is a field.  Further note that $\beta$ is an $\sR$-module homomorphism, hence $H^0(\beta)$ is a $H^0(X, \sR)$-module homomorphism.  Thus $H^0(\beta)$ is a non-zero homomorphism over a field with the target being a one-dimensional vector space. Hence $H^0(\beta)$ is surjective.   This proves that the pair $(X \times_k k^{1/p}, \Delta \times_k k^{1/p})$ is globally $F$-regular.
 \end{proof}

\begin{rem}
\label{rem.GloballyFRegForSameE}
Indeed, the proof of Proposition \ref{prop:globally_F_regular_base_change}
 in fact shows something stronger.  If $D> 0$ is a divisor and $\mathcal{O}_X \to (\mathcal{O}_X( (p^e - 1) \Delta + D))^{1/p^e}$ splits for some $e$ divisible by $a$, then $\sR \to \Big(\sR  \left( h^* (p^e -1) \Delta +  h^* D  \right)\Big)^{1/p^e}$ also splits for the same $e > 0$.
 \end{rem}

\begin{cor}
\label{cor.GloballyFRegBaseChange}
Suppose that $(X, \Delta)$ is a proper, globally $F$-regular variety over an $F$-finite field $k$ such that $H^0(X, \mathcal{O}_X) \supseteq k$ is a separable field extension.  Then $(X,\Delta)$ is geometrically globally $F$-regular.  In other words, for every field extension $k \subseteq K$ such that $K$ is also $F$-finite, we have that $(X \times_k K, \Delta \times_k K)$ is globally $F$-regular.
 \end{cor}
\begin{proof}The proof is essentially the same as Corollary \ref{cor.BaseChangeOfFSplit}.
We use the notation from the proof of Proposition \ref{prop:globally_F_regular_base_change},
in particular $D$ is a divisor whose support contains the support of $\Delta$ and such that $X \setminus D$ is affine and smooth over $k$ and $e$ is such that $(p^e - 1)\Delta$ is integral.  We also set $\sR_e = \mathcal{O}_X \otimes_k k^{1/p^e}$ and $h_e : \Spec \sR_e \to X$ to be the projection. %and $m : X \times_k K \to X$ to be the projection.
Observe that
 \begin{equation}
\label{eq.SplitForNiceE}
\sR_e  \to \big( \sR_e( h^* (p^e - 1)\Delta + h_e^* D \big)^{1/p^e}
 \end{equation}
is split for some $e > 0$ by Remark \ref{rem.GloballyFRegForSameE}.

 But then $\sR_e \to \big(\mathcal{O}_X((p^e - 1)\Delta + D)\big)^{1/p^e}$ is also split since it factors \eqref{eq.SplitForNiceE}.
   Thus there is a map $\phi : \big(\mathcal{O}_X((p^e - 1)\Delta + D)\big)^{1/p^e} \to \sR_e$ which sends $1$ to $1$.  It follows that
 \[
\begin{array}{rcl}
(\phi \otimes_{k^{1/p^e}} K^{1/p^e}) & : & \big(\mathcal{O}_X((p^e - 1)\Delta + D)\big)^{1/p^e} \otimes_{k^{1/p^e} } K^{1/p^e} \\
& \cong & (\mathcal{O}_X((p^e - 1)\Delta + D) \otimes_k K)^{1/p^e} \\
 & \to & \sR_e \otimes_{k^{1/p^e}} K^{1/p^e}
\end{array}
\]
also sends $1$ to $1$.
But obviously since $K \to K^{1/p^e}$ is split, so is the map $\mathcal{O}_X \otimes_k K \to \mathcal{O}_X \otimes_k K^{1/p^e}$ and so there is a map $\gamma : \mathcal{O}_X \otimes_k K^{1/p^e} \to \mathcal{O}_X \otimes_k K$ which also sends $1$ to $1$.  We compose and so obtain:
 \[
 (\mathcal{O}_X((p^e - 1)\Delta + D) \otimes_k K)^{1/p^e}  \xrightarrow{(\phi \otimes_{k^{1/p^e}} K^{1/p^e})} \mathcal{O}_X \otimes_k K^{1/p^e} \xrightarrow{\gamma}  \mathcal{O}_X \otimes_k K.
\]
This implies that $(X \times_k K, \Delta \times_k K)$ is globally $F$-regular as desired.
 \end{proof}

 \begin{rem}
In the case that $X$ is additionally projective, we believe it is possible to prove the previous corollary with a variant of the following argument.  Let $R$ denote the section ring of $X$ with respect to some ample divisor and let $S$ denote the section ring of $X \times_k K$ with respect to the pullback of the same ample divisor.  We then have a map $R \to S$.  Note that $H^0(X, \cO_X) \otimes_k K$ is the spectrum of the fiber of $\Spec S$ over cone point of $\Spec R$.  In particular, the fiber is regular.  At this point, we can apply a variant of the argument of \cite[Lemma 4.5]{SZ}, %Schwede-Zhang Bertini theorems,
which is a generalization of \cite[Section 7]{HHbc}, %Hochster-Huneke  F-regularity, test elements, and smooth base change
use the strong $F$-regularity of $(R, \Delta_R)$ to conclude the strong $F$-regularity of $(S, \Delta_S)$.\end{rem}
Now we come to the proof of our main theorem for this section.

\begin{proof}[Proof of Theorem \ref{general fiber}]
Set $\eta$ to be the generic point of $Y$ with residue field $K$.  Since $(X, \Delta)$ is globally $F$-regular, so is $(X_{K}, \Delta_{K})$ where $X_{K} = X \times_Y \Spec K$ and $\Delta_K$ is the pullback of $\Delta$ along the flat map $X_K \to X$.  Notice that since $f_* \cO_X = \cO_Y$, we have that $f_* \cO_{X_{K}} = K$.  In particular, it immediately follows from Corollary \ref{cor.GloballyFRegBaseChange} that $(X_{K}, \Delta_{K})$ is geometrically globally $F$-regular over $K$ and is hence geometrically normal (so that the general fibers are normal).  Choose $A$ an effective divisor on $X$ containing the support of $\Delta$ such that $X \setminus A$ is affine and smooth over $Y$ ($A$ can be relatively ample if $X \to Y$ is projective).  Using Proposition \ref{prop:globally_F_regular_base_change} and Remark \ref{rem.GloballyFRegForSameE}, we know that for some $e > 0$, the map
\[
\begin{array}{rl}
& \cO_{X_{K^{1/p^e}}} \\
= & \cO_{X_K} \otimes_K K^{1/p^e} \\
\to & \cO_{X_K}^{1/p^e} \otimes_{K^{1/p^e}} K^{1/p^{2e}}\\
 = & (\cO_{X_{K^{1/p^e}}})^{1/p^e} \\
 \to & (\cO_{X_{K^{1/p^e}}}(\lceil (p^e - 1) \Delta_{K^{1/p^e}} \rceil + A_{K^{1/p^e}}) )^{1/p^e}
\end{array}
\]
splits.  This map factors through the twisted relative Frobenius
\[
\begin{array}{rl}
\cO_{X_{K^{1/p^e}}} \to (\cO_{X_K}(\lceil (p^e - 1) \Delta_{K} \rceil + A_K) )^{1/p^e}
\end{array}
\]
which then also splits as well.  By shrinking the base $Y$, we may assume that $Y = \Spec B$ is affine and regular and that $f : X \to Y$ is flat with geometrically normal fibers.  For simplicity of notation, we set $\cO_{X_{B^{1/p^e}}} := \cO_{X \times_{\Spec B} \Spec B^{1/p^e}}$.  Now shrink $Y$ further if necessary so that the map
\begin{equation}
\label{eq.RelativeFrobSplittingToBaseChange}
\begin{array}{rl}
\cO_{X_{B^{1/p^e}}} \to (\cO_{X}(\lceil (p^e - 1) \Delta \rceil + A) )^{1/p^e}
\end{array}
\end{equation}
splits.  Note that since $X$ is normal, $A$ and the support of $\Delta$ are Cartier in codimension $1$ and so they remain Cartier in codimension $1$ over an open set of fibers.  By shrinking $Y$ again, we may assume that $\Delta$ and $A$ are Cartier in codimension $1$ on the fibers and in particular, we will have no trouble restricting them to the fibers if we reflexify (or equivalently take the S2-ification).

For each perfect point $y \in Y$, we tensor the splitting of \eqref{eq.RelativeFrobSplittingToBaseChange} by $\otimes_{B^{1/p^e}} k(y)^{1/p^e}$ and then reflexify with respect to $\cO_{X_y}$ and so obtain a splitting of
\[
\cO_{X_{k(y)^{1/p^e}}} \to (\cO_{X_{k(y)}}(\lceil (p^e - 1) \Delta_{y} \rceil + A_{y}) )^{1/p^e}.
\]
Since the base field is perfect, $X_y = X_{k(y)} = X_{k(y)^{1/p^e}}$.  Also note that $X_y \setminus A_y$ is smooth and $\mathrm{Supp} \Delta_y$ is contained in $\mathrm{Supp} A_y$  and so we have shown that $(X_{y}, \Delta_y)$ is globally $F$-regular by \cite[Theorem 3.9]{SS}.

 \end{proof}

\begin{rem}
One can replace the most of the proof of the above theorem with an application of \cite[Theorem C.(a)]{PSZ}, at least in the case that the map $f : X \to Y$ is projective (instead of merely proper).
%In some ways, using \cite[Theorem C.(a)]{PSZ} in the previous proof is overkill.  Using the same method as in \cite{PSZ} in our very simple special case, it is possible to give a more explicit proof of the above result.  However, we opted for the above proof for the sake of brevity and to highlight \cite{PSZ}.
 \end{rem}

\section{The proof of the main theorem}\label{section-main}

First we collect recent results of the minimal model theory in positive characteristic. For basics on singularities of pairs in positive characteristic and see \cite[Section 2.1 and 2.2]{HX} and \cite[Section 2]{Birkar}. The existence of log resolutions of pairs is known for varieties of dimension three \cite{cp1}, \cite{cp2}, \cite{cut}.

%From the proof of \cite[Theorem 1.8]{CTX} and \cite{Birkar} we can take a Mori fiber space for a threefold with non pseudo-effective canonical divisor as follows:
In this section, we use the following two results of the MMP.  Both of these statement are well-known for experts but for the sake of completeness, we give proofs. Note that the assumption of characteristic in the following theorems is used when we apply the results of \cite{HX}, \cite{Birkar}, \cite{Xu}.

\begin{thm}\label{mfs}
Let $k$ be an algebraically closed field of characteristic $p\geq 7$.
Let $X$ be a projective $\mathbb{Q}$-factorial terminal threefold over $k$.
Assume $K_X$ is not pseudo-effective.
Then there exists a sequence of birational maps
$$X=:X_0\overset{f_0}\dashrightarrow X_1\overset{f_1}\dashrightarrow\cdots\overset{f_{N-1}}\dashrightarrow X_N=: X'$$
which satisfies the following properties.
\begin{enumerate}[(1)]
\item Each $X_i$ is a projective $\mathbb{Q}$-factorial terminal threefold.
\item {Each $f_i$ is a $K_X$-flip or $K_X$-divisorial contraction.} %there exist non-empty open subsets $\tilde X_i\subset X_i$ and
%$\tilde X_{i+1}\subset X_{i+1}$ such that ${\rm codim}_{X_{i+1}}(X_{i+1}\setminus\tilde X_{i+1})\geq 2$ and
%that the restriction $f_{i}|_{\tilde X_i}:\tilde X_i\dashrightarrow \tilde X_{i+1}$ is
%a projective surjective morphism.}
\item{There exists a surjective morphism $g:X'\to Y'$
such that
\begin{itemize}
\item{$g_*\mathcal O_{X'}=\mathcal O_{Y'}$ and $\dim X'>\dim Y',$}
\item{$\rho(X'/Y')=1$, and}
\item{$-K_{X'}$ is ample over $Y'$.}
\end{itemize}}
\end{enumerate}
We call this sequence a $K_X$-MMP.
\end{thm}

\begin{proof}
First, there are no infinite sequences of flips for terminal threefolds.
The proof of this is the same as the one of \cite[6.17]{KM}.
Therefore, it suffices to show that
there exists a $K_X$-flip,  $K_X$-divisorial contraction, or  surjective morphism $g:X\to Y$ which satisfies (3).

Since $K_X$ is not pseudo-effective, $K_X$ is not nef.
Then, by \cite[Theorem 1.8]{CTX},
we can find an ample \Q-divisor $H$ on $X$ such that
$L:=K_X+H$ is nef and that $\overline{NE}(X)\cap L^{\perp}=:R$
is an extremal ray of $\overline{NE}(X)$.
There are the following two cases.
\begin{enumerate}
\item[(a)]{$L$ is big.}
\item[(b)]{$L$ is not big.}
\end{enumerate}

(a)
Assuming $L$ is big,
we show that there exists a $K_X$-flip or a $K_X$-divisorial contraction $h:X\dashrightarrow Z.$
By \cite[Theorem~1.4]{Birkar} or \cite[Theorem 1.1]{Xu}, $L$ is semi-ample.
Let $g:X \to Y$ be the birational morphism, induced by $|mL|$, such that
$g_*{\mathcal O}_X={\mathcal O}_Y$.
If $g$ is a divisorial contraction, then we are done.
Thus we can assume that $g$ is a flipping contraction.
Then the flip exists by \cite[Theorem 1.1]{HX}.

(b)
Assuming $L$ is not big,
we show that there exists a surjective morphism $g:X \to Y$ which satisfies
the properties in (3).
By \cite[Corollary 1.5]{CTX}, rational curves $C$ such that $[C]\in R$ cover $X$.
Then we can apply \cite[4.10]{Kollar} and we obtain a morphism $g:X \to Y$ as desired.
\end{proof}

\begin{thm}\label{special_mmp}
Let $k$ be an algebraically closed field of characteristic $p\geq 7$.
Let $Y$ be a log terminal threefold over $k$.
Let \mbox{$\beta:W \to Y$} be a log resolution of $Y$ and
let $E$ be the reduced divisor on $W$ such that $\Supp\,E={\rm Ex}(\beta)$.
Then there exists a sequence of birational maps
$$W=:W_0\overset{f_0}\dashrightarrow W_1\overset{f_1}\dashrightarrow\cdots\overset{f_{N-1}}\dashrightarrow W_N=: W'$$
$$E=:E_0,\,\,\,\, E_{i+1}:=(f_i)_*E_i,\,\,\,\, E_{N}=:E'$$
which satisfies the following properties.
\begin{enumerate}[(1)]
\item Each $(W_i, E_i)$ is a $\mathbb{Q}$-factorial dlt threefold with a projective birational morphism
$\beta_i:W_i \to Y$.
\item {Each $f_i$ is a $(K_{W_i}+E_i)$-flip or $(K_{W_i}+E_i)$-divisorial contraction over $Y$.}
\item{$K_{W'}$ is nef over $Y$  and $E'=0$.}
\end{enumerate}
We call this sequence a $(K_W+E)$-MMP over $Y$.
\end{thm}

\begin{proof}
First, we show that a $(K_W+E)$-MMP over $Y$ terminates after finitely many steps, that is,
there are no infinite sequences of flips.
Since $Y$ is log terminal,
we can write
$$K_W+E=\beta^*K_{Y}+D\equiv_Y D$$
where $D$ is an effective \Q-divisor such that $\Supp\,D=\Supp\,E=\Ex(\beta)$.
$D_1\equiv_Y D_2$ means that
$D_1\cdot C=D_2\cdot C$
for every curve $C$ on $W$ such that $\beta(C)$ is one point.
Thus, for every step $f_i:W_i \dashrightarrow W_{i+1}$,
the exceptional locus $\Ex(f_i)$ is contained in  $\Supp\, E_i$.
Then, we can apply \cite[Proposition 4.7]{Birkar} and a sequence of flips terminates.

Assume that $(W_i, E_i)$ satisfies the property (1).
We show that
there exists a $(K_{W_i}+E_i)$-flip or  $(K_{W_i}+E_i)$-divisorial contraction over $Y$, or that
$K_{W_i}+E_i$ is nef over $Y$.
Thus, suppose $K_{W_i}+E_i$ is not nef over $Y$.
By \cite[3.3]{Birkar},
we can find a $\beta_i$-ample \Q-divisor $H$ on $W_i$ such that
$L:=K_{W_i}+E_i+H$ is $\beta_i$-nef over $Y$ and that $\overline{NE}(W_i/Y) \cap L^{\perp}=:R$
is an extremal ray of $\overline{NE}(W_i/Y)$.
By \cite[Theorem~1.4]{Birkar} or \cite[Theorem 1.1]{Xu}, $L$ is $\beta_i$-semi-ample.
Then, $L$ induces a projective birational morphism $g:W \to Z$ over $Y$ such that
$g_*{\mathcal O}_W={\mathcal O}_Z$.
If $g$ is a divisorial contraction, then we are done.
Thus we can assume that $g$ is a flipping contraction.
Then the flip exists by \cite[Theorem 1.1]{HX}.

Therefore, we obtain a sequence
$$W=:W_0\overset{f_0}\dashrightarrow W_1\overset{f_1}\dashrightarrow\cdots\overset{f_{N-1}}\dashrightarrow W_N=: W'$$
such that $K_{W'}+E'$ is nef over $Y$.
Next we show $E'=0$.
Indeed we have
$$K_{W'}+E'\equiv_Y D'$$
where $D'$ is the push-forward of $D$.
Therefore, $D'$ is exceptional over $Y'$ and nef over $Y'$.
Then, by the negativity lemma (cf. \cite[2.3]{Birkar}), we obtain $D'=0$.
Since $\Supp\, D=\Supp\, E=\Ex(\beta)$, we obtain $E'=0$.
\end{proof}

We explain now how to take terminalizations  of globally $F$-regular threefolds.
\begin{prop}\label{terminal}
Let $k$ be an algebraically closed field $k$ of characteristic $p\geq 7$.
Let $X$ be a projective globally $F$-regular threefold over $k$.
Then there exists a birational morphism
$\pi:X'\to X$ from
a projective \Q-factorial terminal globally $F$-regular threefold $X'$.
\end{prop}

\begin{proof}
%Take a log resolution $f:Y\to X$.
%Run a generalized $K_Y$-MMP over $X$.
%Then, we obtain a birational morphism
%$\pi:X'\to X$ from
%a projective \Q-factorial terminal threefold $X'$ such that $K_{X'}$ is nef over $X$.
By Theorem \ref{ss1}, there exists an effective $\mathbb{Q}$-divisor $\Delta$ such that $K_X+\Delta$ is \Q-Cartier
and that $(X, \Delta)$ is globally $F$-regular, in particular, Kawamata log terminal.  By \cite[Theorem 1.7]{Birkar},
there exist a birational morphism $\pi:X' \to X$ and an effective $\mathbb{Q}$-divisor $\Gamma$ on $X'$ such that $K_{X'}+\Gamma=\pi^*(K_X+\Delta)$ and that $(X',\Gamma)$ is $\mathbb{Q}$-factorial and terminal.
%Then, $-E$ is $\pi$-nef and $\pi$-exceptional.
%By the negativity lemma, $E$ is effective.
It follows easily that $(X', \Gamma)$ is globally $F$-regular since in fact any splitting used to prove global $F$-regularity of $(X, \Delta)$ also extends to a splitting of $(X', \Gamma)$ (here we crucially use that $\Gamma$ is effective, see for example \cite[Section 7.2]{BS} and \cite[Lemma 3.3]{gt}).
In particular, $X'$ is also globally $F$-regular by Lemma \ref{lem_1.2}.
\end{proof}

Using the MRCC fibration \cite[Ch IV, Section 5]{Kollar2} we show below that \Q-factorial Fano varieties with Picard number one over uncountable fields are rationally chain connected.

\begin{prop}\label{rcc-pic-one}
Let $k$ be an uncountable algebraically closed field.
Let $X$ be a projective normal $\mathbb{Q}$-factorial variety over $k$.
Then, the following assertions hold.
\begin{enumerate}
\item{If $-K_X$ is big, then $X$ is uniruled. }
\item{If $-K_X$ is ample and $\rho(X)=1$, then $X$ is rationally chain connected.}
\end{enumerate}
\end{prop}

\begin{proof}
Let $n:=\dim X$.

(1)
Taking general hyperplane sections $H_i$,
we can find a smooth curve $C=H_1\cap\cdots \cap H_{n-1}$ such that
$C$ does not intersect the singular locus of $X$.
It is sufficient to show that, for every $c\in C$,
there exists a rational curve $R$ passing through $c$.
Since $C$ is a proper curve contained in the non-singular locus $X_{{\rm reg}}$,
we see $K_X\cdot C<0$.
Then, the assertion follows from  Bend and Break (cf. \cite[Ch II, Theorem 5.8]{Kollar2}).

(2)
Take the maximally rationally chain connected (MRCC) fibration \cite[Ch IV, Section 5]{Kollar2}:
$$X\supset X'\overset{f}\to Y.$$
Note that $X'$ is a non-empty open subset of $X$ and
$f$ is a proper surjection.
By (1) and \cite[Ch IV, 5.2.1 Complement]{Kollar2}, we see $\dim Y<\dim X$.
It is enough to show $\dim Y=0$.
Thus, assume $0<\dim Y<\dim X$ and
let us derive a contradiction.
Choose a closed point $y\in Y$ and a non-zero effective Cartier divisor $D_Y$ on $Y$
such that $y\not\in\Supp D_Y$.
Fix a proper curve $C_X$ in the fiber $f^{-1}(y)$.
%Then, we see $C_X\cap f^{-1}(D_Y)=\emptyset.$
Take a prime divisor $D_X$ on $X$
contained in the closure $\overline{f^{-1}(D_Y)}$.
Then we see
$C_X\cap D_X \subset f^{-1}(y) \cap \overline{f^{-1}(D_Y)}=\emptyset.$
On the other hand, since $X$ is \Q-factorial and $\rho(X)=1$,
$D_X$ must be a \Q-Cartier ample divisor.
This is a contradiction.
\end{proof}

First we show the rationality of globally $F$-regular curves and surfaces:
\begin{prop}\label{Rationality of glFsurf}Let $X$ be a projective globally F-regular surface or curve over an algebraically closed field of positive characteristic.
Then $X$ is rational.
\end{prop}

\begin{proof} First we see that $H^{1}(X, \mathcal{O}_X)=0$ by \cite[5.5 Corollary]{Sm}. Thus $X$ is a smooth rational curve when $\mathrm{dim}\,X=1$. Next when $X$ is a surface, by taking the minimal resolution, and noting that the discrepancies are all $\leq 0$, we may assume $X$ is smooth (cf. \cite[Section 7.2]{BS} and \cite[Lemma 3.3]{gt}). We see that $-K_X$ is big by  Theorem \ref{ss1}. Thus $H^0(X, mK_X)=0$ for any $m>0$. Castelnuovo's rationality criterion implies that  $X$ is rational.

\end{proof}

Having finished the preparatory steps we start the  proof of Theorem \ref{0gfr-rcc}.

% \begin{thm}\label{RCC-Freg}
% Let $k$ be an algebraically closed field of characteristic $p\geq 11$.
% Let $X$ be a projective globally $F$-regular threefold over $k$.
% Then, $X$ is rationally chain connected.
% \end{thm}

\begin{prop}
\label{prop:RCC-Freg_1}
Let $k$ be an algebraically closed field of characteristic $p\geq 11$.
Let $X$ be a projective globally $F$-regular threefold over $k$.
Then, to prove that $X$ is rationally chain connected it is sufficient to assume that $X$ is terminal, $\mathbb{Q}$-factorial and it admits a Mori fiber space structure $f : X \to Y$ with $\dim Y=1$ or $2$.
\end{prop}

\begin{proof}

\begin{step}
We may assume $k$ is uncountable by Lemma \ref{base_change_rcc} and Corollary \ref{cor.GloballyFRegBaseChange}.
Moreover by Proposition~\ref{terminal},
we also may assume that $X$ is $\mathbb{Q}$-factorial and terminal.
\end{step}

\begin{step}\label{reduction-mfs}
In this step, we show that we may assume $X$ has a Mori fiber space structure $f: X \to Y$.

We apply Theorem~\ref{mfs} and so run a $K_X$-MMP
$$X=:X_0\overset{f_0}\dashrightarrow X_1\overset{f_1}\dashrightarrow\cdots\overset{f_{N-1}}\dashrightarrow X_N=: X'$$
where $X'$ is also globally $F$-regular by \ref{mfs}(2) and has a Mori fiber space structure $X' \to Y'$.
We claim that $X$ is rationally chain connected if $X'$ is rationally chain connected.

Take a common log resolution
\begin{equation*}
\label{eq:line}
\xymatrix{ & W \ar[dl]_{\alpha} \ar[dr]^{\beta}\\
 X \ar@{.>}[rr] &  & X'.}
\end{equation*}
Then it suffices to show that $W$ is rationally chain connected.
Set $E$ to be the exceptional divisor of $\beta$
 %$$E=\sum_{E^j:\text{$\beta$-exceptional}}E^j, $$
 and run a $(K_W+E)$-MMP over $X'$
$$W=:W_0\overset{f_0}\dashrightarrow W_1\overset{f_1}\dashrightarrow\cdots\overset{f_{N-1}}\dashrightarrow W_N=: W'$$
$$E=:E_0,\,\,\,\, E_{i+1}:=(f_i)_*E_i,\,\,\,\, E_{N}=:E'$$
by Theorem~\ref{special_mmp}.
We see $E'=0$ by Theorem~\ref{special_mmp} (3).
%Since it holds that  $\mathrm{Supp}\,G=E$ for the .quation
% $$K_W+E=\beta^*K_{X'}+G,$$ this MMP terminates from the special termination theorem (see \cite[Section 5]{HX} and \cite[Theorem 4.2.1]{fujino-special}).
% Note that log flips in this program exist by \cite[Theorem 1.1]{HX}.

\begin{claim}
%Since $X'$ is $\mathbb{Q}$-factorial terminal,
The end result $W'$ of this MMP is isomorphic to $X'$.
\end{claim}
\begin{proof}[Proof of claim]
Note first that $X'$ is $\mathbb{Q}$-factorial and terminal.
Since $K_{X'}$ is \Q-Cartier,
we can write $K_{W'}=\gamma^*K_{X'}+D$ where $\gamma:W' \to X'$ and $D$ is $\gamma$-exceptional.
Since $X'$ is terminal, $D$ must be effective.
By the negativity lemma (cf. \cite[2.3]{Birkar}), we see $D=0$.
Since $X'$ is terminal, $\gamma:W' \to X'$ is a small birational morphism.
Assume that $\gamma$ is not an isomorphism.
Then, we can find a $\gamma$-exceptional curve $C$ on $W'$ and a prime divisor $D_{W'}$ on $W'$
which intersects $C$ but does not contain $C$.
Since $W'$ is \Q-factorial, $D_{W'}\cdot C>0$.
On the other hand, since $X'$ is \Q-factorial and $\gamma$ is small,
we obtain $D_{W'}=f^*f_*D_{W'}$, which implies $D_{W'}\cdot C=0$.
This is a contradiction and proves the claim.
\end{proof}

 Let $g:W_i \dashrightarrow W_{i+1}$ be a step of this MMP. Inductively we may assume that $W_{i+1}$ is rationally chain connected. We show the rational chain connectedness of $W_i$ case by case:

 \begin{case} $g$ is a divisorial contraction and the exceptional  divisor is contracted to a  point.
 \end{case}

\begin{case} $g$ is a divisorial contraction and the exceptional   divisor  is contracted  to a curve.
 \end{case}
\begin{case} $g$ is a flip.
\end{case}
Let $F$ be the exceptional prime divisor in Case 1 and Case 2 or a component of $E_i$ such that $C \subseteq F$ in Case 3, where $C$ is some flipping curve. By applying \cite[Theorem 3.1 and Proposition 4.1]{HX} to the pair $(W_i, E_i- \frac{1}{n}(E_i-F))$ for $n\gg1$, it holds that $F$ is normal.
We fix such a large integer $n\gg 1$.
Note that, for $K_F+\Delta_F:=(K_{W_i}+E_i-\frac{1}{n}(E_i-F))|_{F}$, $(F,\Delta_F)$ is klt.

Then, in Case 1, since $-(K_F+\Delta_F)$ is ample,  $F$ is a rational surface, in particular it is rationally connected. Thus $W_i$ is rationally chain connected.

Next, for Case 2, we consider the surjective morphism $h: F \to B$ to a curve $B$ which is induced by the Stein factorization theorem of $g|_{F}$.
Since  $-(K_F+\Delta_F)$ is $h$-ample,
the adjunction formula and the fact that the general fiber of $h$ is reduced (e.g., \cite[7.3]{Badescu}) implies that
a general fiber $D$ of $h$ is $\mathbb P^1$.
Take a resolution of singularities: $F' \to F$ and
we see that the composition morphism $F' \to F \overset{h}\to B$ is a ruled surface structure.
Thus every irreducible component of a closed fiber of $h$ is a rational curve. Therefore  every irreducible component of a closed fiber of $g|_{F}$ is a rational curve. Thus $W_i$ is rationally chain connected.

In Case 3, let $C$ be an arbitrary flipping curve.
We show $C$ is rational.
Note that since $F$ is log terminal, $F$ is \Q-factorial (cf. \cite[Proposition 17.1]{Lipman_Rational_Singularities_with_applications}, \cite[Theorem~14.4]{T2}).
Further note that $(K_F+\Delta_F)\cdot C <0$, $C^2<0$ and $0\leq \mathrm{coeff}_C \Delta_F <1$. In particular, $(K_F + C) \cdot C <0$, which implies by adjunction that $C$ is indeed a (smooth) rational curve.
Thus $W_i$ is rationally chain connected.
\end{step}

\begin{step}
By Step \ref{reduction-mfs}, we can assume that $X$ has a Mori fiber structure $f:X\to Y$.

Since $X$ is a threefold,  $\mathrm{dim}\,Y=0, 1, $ or $2$. If $\mathrm{dim}\,Y=0$, then $X$ is rationally chain connected by Theorem~\ref{rcc-pic-one}. Hence we may further assume that $\mathrm{dim}\,Y=1 $ or $\mathrm{dim}\,Y=2$.
%Then there are three possibilities:
%\begin{enumerate}
%\item[(0)] Assume $\mathrm{dim}\,Y=0$.
%Then, the assertion follows from Theorem~\ref{rcc-pic-one}.

%\item  Assume $\mathrm{dim}\,Y=1$.
%Then Lemma \ref{glFness for MMP} implies $Y\simeq \mathbb{P}^1_k$.
%By Theorem~\ref{general fiber}, general fibers of $f$ are rational.
%By Proposition \ref{dim1-section}, $f:X\rightarrow Y$ has a section.
%Thus, $X$ is rationally chain connected.

%\item  Assume $\mathrm{dim}\,Y=2$. Then, the assertion follows from Theorem~\ref{dim2-rcc}.
%\end{enumerate}
\end{step}

\end{proof}

\subsection{The case where $\mathrm{dim}\,Y=1$.} Here we show the following proposition.

\begin{prop}
\label{prop:RCC-Freg_2}
Let $k$ be an algebraically closed field of characteristic $p\geq 11$.
Let $X$ be a projective globally $F$-regular threefold over $k$.
Then, to prove that $X$ is rationally chain connected we may assume that  $X$ is terminal, $\mathbb{Q}$-factorial and it admits a Mori fiber space structure $f : X \to Y$ with $\dim Y=2$.
\end{prop}

We use the following result of Hirokado, and also a result of de Jong--Starr for the proof. Note that the assumption that $ p \geq 11$ in  Proposition \ref{prop:RCC-Freg_2} (and hence in Theorem \ref{0gfr-rcc}) depends only on Hirokado's Theorem.

\begin{thm}[{\cite[Theorem 5.1(2)]{mh}}]\label{Hirokado}
Let $f: X \to C$ be a proper surjective morphism
from a smooth threefold to a smooth curve with
$f_*\mathcal{O}_X=\mathcal{O}_C$.
Suppose the following conditions.
\begin{enumerate}
\item{$p\geq 11$.}
\item{A general fiber of $f$ is normal.}
\item{The anti-canonical divisor of a general fiber is ample.}
\end{enumerate}
Then, general fibers are smooth.
\end{thm}

\begin{thm}\label{ghs-type}\label{ghs} Let $f:X \to C$ be a proper surjective morphism
from a projective variety to a smooth projective curve $C$. Suppose the general fiber of $f$ is a smooth and separably rationally connected variety, then $f$ has a section.
 \end{thm}

\begin{proof}[Proof of Theorem \ref{ghs}]
Since $X$ is an integral scheme and $C$ is a smooth curve,  $f$ is flat.
Then the assertion follows from \cite[Theorem]{dj-st} (cf.  \cite[Theorem 1.1]{ghs}).
\end{proof}

The above is the famous Graber--Harris--Starr theorem for separably rationally connected varieties. Moreover we use the fact that three dimensional terminal singularities are isolated:
\begin{prop}[{\cite[Corollary 2.30]{Kollar3}}]\label{isolated}Let $X$ be a terminal threefold over an algebraically closed field of any characteristic. Then $X$ has only isolated singularities.
 \end{prop}

\begin{prop}\label{dim1-section}
Let $f:X\to Y$ be a projective surjective morphism between normal varieties with $f_*\mathcal{O}_X=\mathcal{O}_Y$.
Assume the following conditions.
\begin{enumerate}
\item{$p\geq 11$.}
\item{$X$ is a terminal globally $F$-regular threefold.}
\item{$-K_X$ is $f$-ample.}
\item{$\dim Y=1.$}
\end{enumerate}
Then, there is a section of $f$, that is, $s:Y\to X$
such that $f\circ s={\rm id}_Y.$
\end{prop}

\begin{proof}%[Proof of Proposition \ref{dim1-section}]Since $\mathrm{dim}\,Y=1$,  we can freely shrink  $Y$.
By Theorem~\ref{general fiber},
general fibers of $f$ are globally $F$-regular.
In particular, they are normal.
Since terminal singularities are isolated,
by shrinking $Y$ temporarily,
we may assume $X$ is smooth.
Thus, we can apply Theorem~\ref{Hirokado} and
there is a smooth fiber $F$ of $f$.
Then, Theorem~\ref{ghs-type} implies the assertion.
\end{proof}

\begin{proof}[Proof of Proposition \autoref{prop:RCC-Freg_2}]
According to Proposition \autoref{prop:RCC-Freg_1} we may assume that $X$ is terminal, $\mathbb{Q}$-factorial and it admits a Mori fiber space structure $f : X \to Y$ with $\dim Y=1$ or $2$. If  $\dim Y=1$, then
 Lemma \ref{glFness for MMP} implies $Y\simeq \mathbb{P}^1_k$.
By Theorem~\ref{general fiber}, general fibers of $f$ are rational.
By Proposition \ref{dim1-section}, $f:X\rightarrow Y$ has a section.
Thus, $X$ is rationally chain connected and so we may assume that $\dim Y = 2$ in order to prove Theorem \ref{0gfr-rcc}.
\end{proof}

\subsection{The case where $\dim Y=2$} Here we conclude the proof Theorem \ref{0gfr-rcc}.
First we show the following lemma:
\begin{lem}\label{p1-fib-lemma}
Let $f:X\to Y$ be a projective surjective morphism between normal varieties with $f_*\mathcal{O}_X=\mathcal{O}_Y$.
Assume the following conditions.
\begin{enumerate}
\item{$X$ is a globally $F$-regular variety.}
\item{$\dim X-\dim Y=1$. }
\end{enumerate}
Then, there exists a non-empty open subset $Y'\subset Y$ which satisfies the following properties.
\begin{itemize}
\item{Every fiber over $y\in Y'$ is isomorphic to $\mathbb P^1$.}
\item{If $C_Y$ is a projective curve in $Y$ which intersects with $Y'$,
then the normalization morphism
$\nu:\overline{C_Y} \to C_Y$ of $C_Y$ factors through $f|_{f^{-1}(C_Y)}:f^{-1}(C_Y)\to C_Y$:
$$\nu:\overline{C_Y} \to f^{-1}(C_Y)\xrightarrow{f|_{f^{-1}(C_Y)}} C_Y.$$ }
\end{itemize}
\end{lem}

\begin{proof}[Proof of Lemma \ref{p1-fib-lemma}]
We may assume $f$ is flat over $Y'$.
Moreover, by Theorem~\ref{general fiber},
we may assume that every fiber over $Y'$ is $\mathbb P^1$.
Thus, for every projective curve $C_Y$ on $Y$ intersecting $Y'$, the general fibers of
$S_Y:=f^{-1}(C_Y)\to C_Y$ are $\mathbb P^1$.
Take the base change of $S_Y \to C_Y$ to the normalization $\overline{C_Y}$ of $C_Y$ to obtain
$\overline{S_Y} \to \overline{C_Y}$.
Then, a morphism $\overline{S_Y} \to \overline{C_Y}$ is a flat projective morphism whose general fibers are $\mathbb P^1$.
Hence $\overline{S_Y}\to \overline{C_Y}$ has a section by \cite[Theorem]{dj-st} and we are done.
\end{proof}

The rationally chain connectedness when $\mathrm{dim}\,Y=2 $ follows from the following proposition:

\begin{prop}\label{dim2-rcc}
Let $f:X\to Y$ be a surjective morphism between projective normal varieties with $f_*\mathcal{O}_X=\mathcal{O}_Y$.
Assume the following conditions.
\begin{enumerate}
\item{$X$ is a globally $F$-regular threefold.}
\item{$\dim Y=2$. }
\end{enumerate}
Then, $X$ is rationally chain connected.
\end{prop}

\begin{proof}[Proof of Proposition \ref{dim2-rcc}]
We may assume that the base field $k$ is uncountable by  Corollary \ref{cor.GloballyFRegBaseChange} and Lemma \ref{base_change_rcc}.
Then, it is sufficient to show that two general points $x_1$ and $x_2$ are connected by rational curves. Since $X$ is globally $F$-regular, so is $Y$ by Lemma \ref{glFness for MMP}. In particular, $Y$ is a rational surface. Then there is an open set $Y' \subseteq Y$, such that any two points of $Y'$ can be connected with a rational curve. By further restricting $Y'$, we may assume that it also satisfies the conclusions of  Lemma~\ref{p1-fib-lemma}.

Fix then $x_1, x_2\in f^{-1}(Y')$.
Let $y_1:=f(x_1)$ and $y_2:=f(x_2)$.
Lemma~\ref{p1-fib-lemma} implies that
the fibers $f^{-1}(y_1)$ and $f^{-1}(y_2)$ are $\mathbb{P}^1$.
Let $C_Y$ be a rational curve connecting $y_1$ and $y_2$.
Then, $x_1$ and $x_2$ are connected by three rational curves:
$f^{-1}(y_1), f^{-1}(y_2)$ and the image of $\overline{C_Y} \to f^{-1}(C_Y)$ as in Lemma~\ref{p1-fib-lemma} where
$\overline{C_Y}$ is the normalization of $C_Y$.
\end{proof}

\begin{proof}[Proof of Theorem \ref{0gfr-rcc}]
Combining Propositions \ref{prop:RCC-Freg_2} and \ref{dim2-rcc} completes the proof.
\end{proof}

\section{On separable rational connectedness}\label{section-src}

In this section, we work over an algebraically closed field $k$ of characteristic $p>0$. In Section \ref{section-main} we showed that  projective globally $F$-regular threefolds are rationally chain connected. Here we address the question whether they are separably rationally connected as well. Recall from Definition \ref{dfn.rationallyConnected} that the latter is a stronger notion.  Since the separably rationally conneted property  is preserved under birational maps, by running the MMP  one can reduce the question to the case of a  threefold which  is either $\mathbb{Q}$-Fano or admits a Mori fiber space structure over a curve or a surface. Here we only consider the second case and the main result is the following:

\begin{thm}\label{dim1-section-2}
Let $f:X\to Y$ be a projective surjective morphism between normal varieties over $k$ with $f_*\mathcal{O}_X=\mathcal{O}_Y$.  Assume that the following conditions hold.
\begin{enumerate}
\item{$p\geq 11$.}
\item{$X$ is a terminal globally $F$-regular threefold.}
\item{$-K_X$ is $f$-ample.}
\item{$\dim Y=1$ or $\dim Y=2.$}
\end{enumerate}
Then $X$ is separably rationally connected.
\end{thm}

\begin{rem} For $\mathbb{Q}$-Fano varieties over fields with positive characteristic, the separable rational connectedness is still unknown. Even for Fano varieties, this  question remains mysterious.  Recently, there has been some progress in this direction, for which we refer the reader to \cite{Zh},\cite{CZ} and \cite{Ti}.
\end{rem}

The proof of Theorem~\ref{dim1-section-2} comes from smoothing combs on families of separably rationally connected varieties. Let us briefly recall the  definitions of a comb and its smoothing.

\begin{dfn}\label{defcomb}
A comb  with $n$ {\it teeth} over $k$ is a projective curve $C$
with $n+1$ irreducible components $C_0,C_1,\dots,C_n$
such that the curves $C_1,\dots,C_n$ are smooth rational curves
disjoint from each other, and each of them meets $C_0$
transversely in a single smooth point of $C_0$. The curve $C_0$ is called the {\it handle} of the comb,
and $C_1,\dots,C_n$ the {\it teeth}.
\end{dfn}
\begin{dfn}
Let $g:C\rightarrow X$ be a morphism from a comb $C$ to a smooth variety $X$ over $k$. A {\it smoothing} of $g$ is a family $\Sigma\rightarrow T$ over a pointed curve $(T, 0)$ together with a morphism $G : \Sigma\rightarrow X$ such that $\Sigma_0\cong C$, $G|_{\Sigma_0}=g$ holds over $0 \in T$ and  $\Sigma_t$
is smooth for general $t\in T$.
\[
\xymatrix{
\\
C \cong \Sigma_0 \ar[d] \ar@/^2pc/[rr]^g \ar@{^{(}->}[r] & \Sigma \ar[r]_-G \ar[d] & X \\
\{0\} \ar@{^{(}->}[r] &  T
}
\]
\end{dfn}

%In particular, we define the very free locus as:

%\begin{dfn}[{\cite[IV. 3.9.4 Theorem]{Kollar2}}]
%Let $X$ be a smooth quasi-projective variety over an algebraically closed field $k$, we define  %Obviously very free locus is contained in smooth locus, $X_{vf} \subset X_{sm}$.\end{dfn}

The following is the main fundamental results about separably rationally connected varieties.
\begin{thm}[{\cite[IV. Theorem 3.7]{Kollar2}} ]
\label{thm:SRC_fundamental}
Let $X$ be a smooth  variety over $k$. Then $X$ is SRC if and only if there is a very free curve $g: \mathbb{P}_k^1 \to X$ in $X$,  i.e. $g^{*}T_X$ is ample. %In particular, when $X$ is projective, every point on $X$ is  lying on a very free curve if $X$ is SRC.
\end{thm}

To begin with the machinery, we present a smoothing of comb theorem for a 1-dimensional family of separably rationally connected varieties, it was already implicitly stated in the celebrated paper \cite{ghs} and \cite{HT}.
\begin{thm}\label{rel}Let $\pi:\mathfrak{X}\to C$ be a surjective morphism defined over $k$ from a projective variety $\mathfrak{X}$ to a smooth projective curve $C$ such that the general fibers of $\pi$ are separably rationally connected. Let $g_0:C' \to  \mathfrak{X}$ be a multisection of $\mathfrak{X}\rightarrow C$.Suppose that the image of $g_0$ lies in the smooth locus of $\mathfrak{X}\rightarrow C$ and meets with very free curves on general fibers of $\pi$. Then for any $d\in\mathbb{Z}$, there exist $q\gg0$ rational curves $g_i: C_i \to \mathfrak{X}, 1 \leq i \leq q$, such that\begin{enumerate}\item For $ 1 \leq i \leq q$,  $g_i:C_i\rightarrow \mathfrak{X}$ factors through a separably rationally connected fiber  $\mathfrak{X}_{c_i}$ for some $c_i\in C$,
and $g_i:C_i\rightarrow \mathfrak{X}_{c_i}$ is very free in $\mathfrak{X}_{c_i}$. \item $\hat{C}=C'\cup C_1\cup...\cup C_q$ is a comb with $q$ teeth  endowed with a natural morphism $g: \hat{C} \to  \mathfrak{X}$.  Furthermore, there is a smoothing  of $g$ for some pointed curve $(T,0)$, denoted  by $\Sigma \to T, G: \Sigma \to \mathfrak{X}$.\item For general $t\in T$, $H^1(\Sigma_t, G_t^*{T_{\mathfrak{X}/C}}\otimes M)=0$ for any line bundle $M$ on $\Sigma_t$ of degree $d$, where $G_t: \Sigma_t \to \mathfrak{X}$ denotes the restriction of $G$ to the fiber $\Sigma_t$.\end{enumerate}
\end{thm}
\begin{rem} It is possible to get a potentially shorter proof of this \emph{relative smoothing} theorem by almost the same argument as in the proof of \cite[Proposition 2.4]{TZ}.  Furthermore, if we apply this theorem to the case of a trivial product $C \times X$ as a family over $C$, where $C$ is a smooth curve inside $X$, and take the multisection to be the diagonal isomorphic to $C$, then \cite[Proposition 2.4]{TZ} will follow as a corollary of this theorem.  Note \cite[Proposition 2.4]{TZ} was useful for several important questions concerning topology and arithmetic of separably rationally connected varieties in the past.   We encourage interested readers to find more applications of arguments using the trivial product in similar ways.
\end{rem}
%\begin{proof} The proof is basically the same as \cite[Proposition 2.4 ]{TZ}, so we only sketch the proof.
%By \cite[Lemma 25]{HT} and \cite[Proposition 24]{HT}, one can attach  sufficiently many very free curves $C_i,i=1,\ldots,q$ on the fibers of $\pi$ to $C'$ such that there is a smoothing of the comb $\hat{C}=C'\cup C_1\cup...\cup C_q$ whose handle is $C'$. Then there is a natural morphism $g:\hat{C}\rightarrow \mathfrak{X}$.  Moreover, one can choose $q$ large enough such that $q-h^1(C', (g^\ast T_{\mathfrak{X}})|_{C'})\gg 0 $. Then $H^1(\Sigma_t, G_t^*{T_{\mathfrak{X}/C}}\otimes M)=0$ for  any line bundle $M$ on $\Sigma_t$ of degree $d$ by \cite[Lemma 2.5]{TZ} and \cite[Lemma 2.6]{TZ}.\end{proof}

%As a corollary, one can easily obtain the following result.

%\begin{conje}Let $X$ be an $n$-dimensional globally $F$-regular variety over an algebraically closed field of positive characteristic $p$. Then there is a constant $c(n)$ depending only on $n$ such that $X$ is separably rationally connected if $p >c(n)$.\end{conje}

The following proposition uses Theorem \ref{thm:SRC_fundamental} and Theorem \ref{rel} (or equivalently and referring to more details-the results of \cite{HT}).  We note here that this proposition is based on discussions of Zsolt Patakfalvi with Zhiyuan Li and Professor Brendan Hassett. We also note that the key difficulty is that one cannot assume resolution of singularities for dimension higher than $3$ -- we only assumed the total space is smooth.

\begin{prop}\label{dim1-src}
%Let $k$ be an algebraically closed field of positive characteristic.
Let $f:X \to Y$ be a surjective morphism such that  $X$ and $Y$ are smooth projective varieties over $k$ such that $Y$ is SRC and also that the geometric generic fiber of $f$ is smooth, irreducible and SRC.  Then $X$ is SRC.
%In particular, when $p \geq 11$, any globally F-regular threefold is either birational to a Q-Fano threefold or separably rationally connected.
\end{prop}

\begin{proof}
%First note that similarly to Lemma \ref{base_change_rcc} we may assume that the base field $k$ is uncountable.
  Choose a general unramified very free curve $\mathbb{P}^1 \to Y$ and define $Z$ to be the component of $X \times_Y \mathbb{P}^1$ dominating $\mathbb{P}^1$ via the projection to the second factor giving $Z$ the reduced induced scheme structure. Let $h : Z \to \mathbb{P}^1$ be the induced map. Since the general fiber of $X \times_Y \mathbb{P}^1 \to \mathbb{P}^1$ is smooth and irreducible, $Z$ is integral.  According to  Theorem~\ref{ghs}, there is  a section $\mathbb{P}^1 \to Z$ of $h$. This defines a rational curve $\sigma : \mathbb{P}^1 \to X$.   Choose a point $r$ on $\mathbb{P}^1$ such that $\sigma(r)$ lies both in the smooth locus of $\sigma(\mathbb{P}^1)$ and in a smooth fiber of $f$.  From here we follow the argument of \cite[Lemma 25]{HT} to show that

\emph{Claim:} there exists a comb $C$ in $X$ the handle of which is the aforementioned $\mathbb{P}^1$ (the domain of $\sigma$) and which is endowed with a map $g: C \to X$, such that
\begin{itemize}
\item $g$ is an immersion,
\item the image of $g$ is nodal,
\item  $g(r)$ is contained in a smooth point of the image of $g$ and
\item by  denoting the normal bundle of $g$ by $N_g$, the restriction of $N_g (-r)$ to every irreducible component of $C$ is globally generated and without higher cohomology ($N_g$ is defined to be the dual of $\ker (g^* \Omega_X \to \Omega_C)$, it is a vector bundle according to  \cite[page 182]{HT}).
% Note that by \cite[Lemma 21]{HT} this condition implies that $N_g (-r)$ itself is also globally generated and without higher cohomology.
\end{itemize}
The argument to prove the above claim is as follows:

\emph{Initial considerations:} Since we chose our initial very free curve in $Y$ generally, the general fiber of $h$ is smooth, irreducible and SRC. In particular, we may choose  $m$ very free curves in general fibers of $h$ meeting $C_0 := \sigma(\mathbb{P}^1)$ transversely. This way we obtain a comb $C =  \bigcup_{i=0}^m C_i$, where all the $C_i$ are smooth rational curves meeting  $C_0$ in nodal singularities and $C_1,\dots, C_m$ do not intersect each others, and we further obtain an immersion $g : C \to X$ with  nodal image such that $r$ is contained in a smooth point of the image of $g$. We claim then that $N_g (-r)|_{C_i}$ is globally generated and without higher cohomology for every $i$ if we choose $m$ to be big enough. Since each $C_i$ is a smooth rational curve, this is equivalent to showing that there is a big enough $m$ such that for all $C_i$, the vector bundle $N_g (-r)|_{C_i}$ is nef . There are two possibilities:

\emph{The case when $C_i$ is a tooth:} Since $r$ does not lie in a tooth, here we have $N_g(-r)|_{C_i} \cong N_g|_{C_i}$. Further, since $C_i$ lies in one of the smooth fibers (of both $f$ and $h$), there is an exact sequence of vector bundles as follows.
\begin{equation*}
\xymatrix{
0 \ar[r] & g^* T_{X/Y}|_{C_i} \ar[r] & g^* T_X |_{C_i}  \ar[r] & g^* f^* T_Y |_{C_i} \ar[r] & 0
}
\end{equation*}
Since $g(f(C_i))= \mathrm{pt}$, the last term is a trivial bundle, hence it is nef. Further the first term is ample (and hence nef) since $C_i$ was chosen to be very free in its fiber. Therefore, $g^* T_X |_{C_i}$ is nef. Since $N_g|_{C_i}$ contains a full rank subbundle which is a quotient of $g^* T_X |_{C_i}$, $N_g|_{C_i}$ is nef as well. This concludes the proof of this case.

\emph{The case when $C_i$ is the handle, i.e. $C_i=C_0$.} Here the situation is much more intricate. Note that $N_g$ is defined as the dual of  $\mathcal{E} := \ker \left( g^* \Omega_X \to \Omega_C \right)$ (since  $g$ is an immersion this homomorphism is surjective). In particular, to prove our goal that $N_g(-r)|_{C_0}$ is nef it is enough to prove that $\mathcal{E}|_{C_0}$ is anti-ample. Note that since $C_0 \cong \mathbb{P}^1$, $\mathcal{E}|_{C_0} \cong \bigoplus_{j=1}^n \mathcal{O}_{\mathbb{P}^1}(a_j)$ (where $n = \dim X$) and hence showing that $\mathcal{E}|_{C_0}$ is anti-ample is equivalent to showing that $a_i <0$, which is further equivalent to showing that $H^0(C_0, \mathcal{E}|_{C_0})=0$. So, assume the contrary and fix any $0 \neq s \in H^0(C_0, \mathcal{E}|_{C_0})=0$.  Consider then the following diagram, where the rows and columns are exact.
\begin{equation*}
\xymatrix{
 & & 0 \ar[d] \\
%& \bigoplus \mathcal{O}_{\mathbb{P}^1}(a_j) \ar[d]^{\cong}
& & g^* f^* \Omega_Y|_{C_0} \ar[d] \\
0 \ar[r] & \mathcal{E}|_{C_0} \ar[r] & g^* \Omega_X|_{C_0} \ar[r] \ar[d] & \Omega_C|_{C_0} \ar[r]  & 0 \\
& & g^* \Omega_{X/Y}|_{C_0} \ar[d] \\
& & 0
}
\end{equation*}
Here one has to argue that $g^* f^* \Omega_Y|_{C_0}  \to g^* \Omega_X|_{C_0}$ and $\mathcal{E}|_{C_0} \to g^* \Omega_X|_{C_0}$ are injective. The reason in  both cases is that the source sheaf is locally free and the map is injective generically.

Since by construction $C_0$ maps onto a very free curve of $Y$, $g^* f^* \Omega_Y|_{C_0}$ is anti-ample and in particular, $H^0(C_0,g^* f^* \Omega_Y|_{C_0})=0$. Hence by the above commutative diagram, $H^0(C_0, \mathcal{E}|_{C_0})$ injects into $H^0(C_0,g_* \Omega_{X/Y}|_{C_0})$.  Therefore, we may fix a smooth closed point $P$ of $C_0$ such that $P$ is contained in a smooth, irreducible SRC fiber of $f$ and such that the image of $s$ in $\Omega_{X/Y} \otimes k(P)$ is not zero. Since the above diagram identifies,  $\mathcal{E}|_{C_0}$ with $\frac{\mathcal{I}_{X/C,P}}{\mathcal{I}_{X/C,P}^2} \subseteq   \left(g^* \Omega_X|_{C_0}\right)_P$, we may find a local function $x_2 \in \mathcal{I}_{X/C,P}$, such that $dx_2$ identifies with $s_p$. Note that although the choice of $x_2$ is not unique, different choices yield the same $dx_2$. So, fix $x_2$. Note also that $x_2$ gives local coordinate also on the fiber $f^{-1} (f(p))$ of $P$ (this follows from the fact that the image of $s_p$ is not zero in $g^* \Omega_{X/
Y}|_{C_0}  \otimes k(P)$). Further we can extend this to local coordinates $x_1, \dots, x_n$  of $X$ at $P$  satisfying the following two conditions:
\begin{enumerate}
\item $x_1, x_3, x_4, \dots, x_r$ are coordinates of $f(P)$ and
 \item $x_2, \dots, x_n$ generate the ideal of $C_0$ at $P$.
\end{enumerate}
According to \cite[Thm 2.42]{Debarre}, we may choose a very  free curve $\tilde{C}$ of the fiber $X_{f(P)}$ of $P$ such that  $\tilde{C}$  is tangent to $x_1=x_3= \dots = x_n=0$ at $P$.  Changing $x_1,x_3, \dots, x_n$ by elements of $m_{X,P}^2$ we may assume that $x_1=x_3= \dots = x_n=0$ is actually a local equation for $\tilde{C}$.  Let us attach then $\tilde{C}$ to $C$ and thus obtain a new comb $C'$ with a map $g' : C' \to X$. In particular then $I_{C'/X,P}=(x_1 \cdot x_2, x_3, \dots, x_n)$. A local computation shows then that $\ker \left( g^* \Omega_X|_{C_0} \to \Omega_C|_{C_0}  \right)$ at $P$  is the free $\mathcal{O}_{C_0,P}$-module  generated by $dx_2,\dots, dx_n$, whereas $\ker \left( g^* \Omega_X|_{C_0} \to \Omega_{C'}|_{C_0}  \right)$ at $P$  is the free $\mathcal{O}_{C_0,P}$-module  generated by $tdx_2, dx_3, \dots, dx_n$, here $t$ is the local uniformizer of $C_0$ at $P$. So, $s$ is not a section of the subsheaf  $\ker \left( g^* \Omega_X|_{C_0} \to \Omega_{C'}|_{C_0}  \right)$ of $\mathcal{E}|_{C_0}$.
Hence the space of sections became one dimension smaller by attaching $\tilde{C}$. In particular, for $m \gg 0$ it is zero dimensional, which is a contradiction. This finishes the proof of the second case and of the claim as well.

\emph{Applying the claim:} According to  \cite[Proposition 24]{HT} our claim implies that there exists a  smoothing $h_t : \Sigma_t \to X$ of $h_0=g$ going through $g(r)$ such that $h_t$ is an immersion for every $t \in T$. That is, there is a section $\lambda : T \to \Sigma$ of $\Sigma \to T$, such that $h_t( \lambda(t)) = r$ for every $t \in T$ (note that a priori in positive characteristic there might only be a closed subvariety $S:=h^{-1}(g(r))$ mapping  finitely (and inserparbly) onto T, however we can replace $\Sigma \to T$ by $\Sigma \times_T S \to S$, which has the same fibers but it has  also a section mapping onto $g(r)$). By restricting $T$ we may assume that the image of this section lies in the relative smooth locus of $\Sigma \to T$.
% Let $F$ be the divisor class of a general fiber of $f$. Then $1 =  C \cdot g^*F =\Sigma_t \cdot h_t^* F$. So, in particular, for general $t$, $h_t$ is a closed embedding of a smooth curve intersecting $F$ in order $1$. Hence $h_t(\Sigma_t)$ is a section of $f$ for general $t \in T$.

Note that since $N_g(-r)$ is nef, $N_g(-2r)$ also has no higher cohomology. Indeed, consider the long exact sequence of cohomology associated to the short exact sequence
\begin{equation*}
\xymatrix{
0 \ar[r] & N_g(-2r) \ar[r] & N_g(-r) \ar[r] & N_g(-r)|_r \ar[r] & 0
}
\end{equation*}
Since $N_g(-r)$ is globally generated, $H^0(C,N_g(-r)) \to H^0(r,N_g(-r)|_r)$ is surjective, and hence $H^1(C,N_g(-2r))=0$ holds. It follows then by a standard argument that $H^1(\Sigma_t, N_{h_t}(-2 \lambda(t)))=0$ for general $t \in T$ (the basic idea is that the flat limit of $N_{h_t}^*$ at $t=0$ is the extension of $N_g^*$ by a sheaf with zero dimensional support). However then since $\Sigma_t \cong \mathbb{P}^1$, this implies that $N_{h_t}$ is ample. Thus by the following exact sequence, so is $h_t^* T_X$.
\begin{equation*}
\xymatrix{
0 \ar[r] &  T_{\Sigma_t} \cong T_{\mathbb{P}^1} \cong \mathcal{O}_{\mathbb{P}^1}(2) \ar[r] & h_t^* T_X \ar[r] &  N_f|_{\Sigma_t} \ar[r] \ar[r] & 0
}
\end{equation*}
Hence $h_t : \Sigma_t  \to X$ is a very free curve and then by Theorem \ref{thm:SRC_fundamental}  $X$ is SRC.

\end{proof}

\noindent{\it Proof of Theorem \ref{dim1-section-2}.} Similar to the proof of Theorem \ref{0gfr-rcc},  here we show the separable rational connectedness case by case:
\begin{itemize}
\item If $\dim Y=1$, then $Y\simeq \mathbb{P}_k^1$ since $Y$ is globally $F$-regular by Lemma~\ref{glFness for MMP}. Moreover, by Theorem~\ref{general fiber}, we know that
general fibers of $f$ are globally $F$-regular and hence normal and rational.
Since terminal singularities are isolated,
we can apply Theorem~\ref{Hirokado} to obtain that
general fibers are smooth. Now let us take a resolution $X' \to X$ that is an isomorphism outside of the singular locus of $X$, and let $f' : X' \to Y$ be the induced morphism. Since terminal threefold singularities are isolated by Proposition \ref{isolated}, $X' \to X$ is isomorphism over the generic point of $Y$, and hence the geometric generic fibers of $f$ and $f'$ agree. However, by our assumptions the geometric general fiber of $f$ has ample anti-canonical class, and we have seen that it is also smooth. Therefore, it is a rational surface and therefore it is SRC. In particular, we may apply Proposition \ref{dim1-src} to $f'$. This shows that $X'$ is SRC and then so is $X$, since it is birational to $X'$.

\item If $\dim Y=2$, then $Y$ is a globally $F$-regular surface, which  has to be rational and then consequently SRC. Further, by Theorem~\ref{general fiber}, the geometric generic fiber of $f$ is globally $F$-regular, and hence it is a projective line, which is also SRC. Proposition \ref{dim1-src} then shows that $X$ is SRC in this case as well.

% So there is a birational map $\mathbb{P}_k^2 \dashrightarrow Y$. Then there is a birational map $X' \dashrightarrow X$ where $X'$ is a conic bundle over $\mathbb{P}_k^2$. Then our assertion follows if the conic bundle $X'$ is SRC. It is not very hard to see that $X'$ is SRC and we only sketch the proof. One can take the inverse image $D\subseteq X'$ of a  line on $\mathbb{P}_k^2$ so that $D$  has at worst isolated singularities.  Then $D$ is obviously SRC. The normal sheaf of $D$ is positive on some very free curve $C$ of $D$. This indicates that $C$ is also a very free curve of $X'$.

\end{itemize}
\qed
\begin{rem}
Since the Mori fiber space here is of dimension $3$, so resolution for surfaces and threefold in positive characteristics can applied to get an easier proof. Nonetheless, the readers are encouraged to find the full implication of Proposition \ref{dim1-src} for more general and higher dimensional questions along this direction.
\end{rem}

\section{RC-ness of varieties of globally $F$-regular type}\label{section-gfr-type}

In this section we show that a globally $F$-regular type variety is rationally connected.

We briefly explain how to reduce things from characteristic zero to characteristic $p > 0$.
%We use the reduction modulo to positive characteristic in only Section \ref{section-gfr-type}.
The reader is referred to \cite[Chapter 2]{HH} and \cite[Section 3.2]{MS} for additional details.

Let $X$ be a normal variety over an algebraically closed field $k$ of characteristic zero and $D=\sum_i d_i D_i$ a $\mathbb{Q}$-divisor on $X$.
Choosing a suitable finitely generated $\mathbb{Z}$-subalgebra $A$ of $k$,
we can construct a scheme $X_A$ of finite type over $A$ and closed subschemes $D_{i, A} \subsetneq X_A$ such that
there exists isomorphisms
\[\xymatrix{
X \ar[r]^{\cong \hspace*{3em}} &  X_A \times_{\Spec \, A} \Spec \, k\\
D_i \ar[r]^{\cong \hspace*{3em}} \ar@{^{(}->}[u] & D_{i, A} \times_{\Spec \, A} \Spec \, k. \ar@{^{(}->}[u]\\
 }\]
Note that we can enlarge $A$ by localization, in particular inverting a single nonzero element, and replacing $X_A$ and $D_{i,A}$ with the corresponding open subschemes.
Thus, applying the generic freeness \cite[(2.1.4)]{HH}, we may assume that $X_A$ and $D_{i, A}$ are flat over $\Spec \, A$.
 Enlarging $A$ further if necessary, we may also assume that $X_A$ is normal and $D_{i, A}$ is a prime divisor on $X_A$.
Letting $D_A:=\sum_i d_i D_{i,A}$, we refer to $(X_A, D_A)$ as a \textit{model} of $(X, D)$ over $A$.

Given a closed point $\mu \in \Spec \, A$, we denote by $X_{\mu}$ (resp., $D_{i, \mu}$) the fiber of $X_A$ (resp., $D_{i, A}$) over $\mu$.
Then $X_{\mu}$ is a scheme of finite type over the residue field $k(\mu)$ of $\mu$, which is a finite field.
Enlarging $A$ if necessary, we may assume that  $X_{\mu}$ is a normal variety over $k(\mu)$, $D_{i, \mu}$ is a prime divisor on $X_{\mu}$ and consequently $D_{\mu}:=\sum_i d_i D_{i, \mu}$ is a $\mathbb{Q}$-divisor on $X_{\mu}$ for all closed points $\mu \in \Spec \, A$.

%Let $\Gamma$ be a finitely generated group of Weil divisors on $X$.
%We then refer to a group $\Gamma_A$ of Weil divisors on $X_A$ generated by a model of a system of generators of $\Gamma$ as a \textit{model} of $\Gamma$ over $A$.
%After enlarging $A$ if necessary, we denote by $\Gamma_{\mu}$ the group of Weil divisors on $X_{\mu}$ obtained by restricting divisors in $\Gamma_A$ over $\mu$.

Given a morphism $f:X \to Y$ of varieties over $k$ and models $(X_A, Y_A)$ of $(X, Y)$ over $A$,  after possibly enlarging $A$, we may assume that $f$ is induced by a morphism $f_A :X_A \to Y_A$ of schemes of finite type over $A$.
Given a closed point $\mu \in \Spec \, A$, we obtain a corresponding morphism $f_{\mu}:X_{\mu} \to Y_{\mu}$ of schemes of finite type over $k(\mu)$.
If $f$ is projective (resp. finite), after possibly enlarging $A$, we may assume that $f_{\mu}$ is projective (resp. finite) for all closed points $\mu \in \Spec \, A$.

\begin{dfn}\label{type}
Use the notation as before.
\begin{enumerate}[(i)]
\item A projective variety (resp. an affine variety) $X$ is said to be of \textit{globally $F$-regular type} (resp. \textit{strongly $F$-regular type}) if for a model of $X$ over a finitely generated $\mathbb{Z}$-subalgebra $A$ of $k$, there exists a (Zariski) dense open subset $S \subseteq \mathrm{Spec} \, A$ of closed points such that $X_{\mu}$ is globally $F$-regular (resp. strongly $F$-regular) for all $\mu \in S$.
\item A projective variety (resp. an affine variety) $X$ is said to be of \textit{dense globally $F$-split type} (resp. \textit{dense $F$-pure type}) if for a model of $X$ over a finitely generated $\mathbb{Z}$-subalgebra $A$ of $k$, there exists a (Zariski) dense subset $S \subseteq \Spec \, A$ of closed points such that $X_{\mu}$ is globally $F$-split (resp. $F$-pure) for all $\mu \in S$.
\end{enumerate}
\end{dfn}

\begin{rem}
(1) The above definition is independent of the choice of a model.

(2) If $X$ is of globally $F$-regular type (resp. strongly $F$-regular type), then we can take a model $X_A$ of $X$ over some $A$ such that $X_{\mu}$ is globally $F$-regular (resp. strongly $F$-regular) for all closed points $\mu \in \Spec \, A$.
\end{rem}

The following proposition is well-known \cite{hw,SS}, we recall it for the convenience of the reader.:

\begin{prop}\label{just singularities}
Let $X$ be a normal projective variety over an algebraically field of characteristic zero.
\begin{enumerate}[$(1)$]
\item If $X$ is $\mathbb{Q}$-Gorenstein and of globally $F$-regular type $($resp. dense globally $F$-split type$)$, then it has only Kawamata log terminal singularities $($resp. log canonical singularities$)$, \cite[Theorem 3.9]{hw}.
\item If $X$ has some effective $\mathbb{Q}$-divisor $\Delta$ such that $(X,\Delta)$ is Kawamata log  terminal and $-(K_X+\Delta)$ is ample, then $X$ is of globally $F$-regular type, \cite[Theorem 5.1]{SS}.
\end{enumerate}
\end{prop}

And we use the following easy lemma:

\begin{lem}\label{triviality}Let $X$ be a normal projective variety over algebraically closed field $k$ and $L$ line bundle on $X$. Then $L \equiv 0$ if $L$ and $L^{-1}$ are pseudo-effective.

\end{lem}

\begin{proof}First we show it when $k$ is uncountable.  Assume by contraction  that there exists a curve  $C$ with $C \cdot L \not =0$. By symmetry, we may assume that $(L \cdot C)>0$. We take a very ample line bundle $H$ on $X$ such that there exist $D_i \in |H|$ for $i=1, \cdots, \mathrm{dim}\,X-1 $ such  that $C \subseteq \bigcap_i D_i $ and $(L \cdot  \bigcap_i D_i) >0$. Note that since $k$ is uncountable, $\bigcap_i D_i$ and $\bigcap_i D_i-C$ is movable curve on $X$ in the sense of \cite[(vi) 1.3 Definition ]{bdpp}. Since $L^{-1}$ is also pseudo-effective, $(L^{-1} \cdot  \bigcap_i D_i)  \geq0$. This is a contradiction to that $(L^{-1} \cdot  \bigcap_i D_i)=-(L \cdot  \bigcap_i D_i)<0$. Next, when $k$ is countable, we take a filed extension $k \subset K$ such that $K$ is uncountable. Note that there are  a very ample divisor A and infinite many $m$ such that $H^0(X, mL+A)$ is not zero by \cite[Proposition 2.8 and Remark 2.9]{CMHS}. Thus after the base extension,  $H^0(X_K, mL_K+A_K)$ ($\simeq  H^0(X,
mL+A) \otimes_k K$) is also not zero, where $X_K, L_K$, and $A_K$ are the base change of $X, L$, and $A$ respectively. Note that $A_K$ is very ample. Then $L_K$ is pseudo-effective.  We apply this arguments for $L^{-1}$. Thus we see that $L_K$ is numerical trivial since $K$ is uncountable. Thus $L$ is also numerical trivial since  $L_K \cdot C_K=L \cdot C$ holds  for any curve $C$ on $X$.

\end{proof}

Next, we introduce the concept of a \emph{dlt blow-up}. The following theorem was originally proven by Hacon \cite[Theorem 3.1]{kk} and a simpler proof was given by Fujino \cite[Theorem 4.1]{fujino3}:

\begin{thm}[Dlt blow-up]\label{dltblowup}
Let $X$ be a normal quasi-projective variety and
$\Delta$ an effective $\mathbb Q$-divisor on $X$ such
that $K_X+\Delta$ is $\mathbb Q$-Cartier. Suppose that $(X,\Delta)$ is lc.
Then there exists a projective birational
morphism $\varphi:Y\to X$ from a normal quasi-projective
variety with the following properties:
\begin{itemize}
\item[(i)] $Y$ is $\mathbb Q$-factorial,
\item[(ii)] $a(E, X, \Delta)= -1$ for every
$\varphi$-exceptional prime divisor $E$ on $Y$, and
\item[(iii)] for $$
\Gamma=\varphi^{-1}_*\Delta+\sum _{E: {\text{$\varphi$-exceptional}}}E,
$$ it holds that  $(Y, \Gamma)$ is dlt and $K_Y+\Gamma=\varphi^*(K_X+\Delta)$.
\end{itemize}

\end{thm}

The following lemma begins the proof of Theorem \ref{0gfrtype-rc}:

\begin{lem}\label{pseudo-effective}Let $X$ be a $\mathbb{Q}$-Gorenstein normal projective variety of dense globally $F$-split type. Then $-K_X$ is pseudo-effective. Moreover, if $X$ is of globally $F$-regular type, $K_X$ is not pseudo-effective.\end{lem}

\begin{proof}Assume that $-K_X$ is not pseudo-effective in order to obtain a contradiction. Then there exists a movable curve $C$ such that $K_X.C>0$. Taking a model $$\mathcal{X} \to \mathrm{Spec}\, A$$ for reduction to positive characteristic, there exists a dense  subset of closed points $S \subseteq \mathrm{Spec}\, A$ such that $X_{\mu}$ is globally $F$-split and $C_{\mu}$ is also a movable curve of $X_{\mu}$ for any closed point $\mu \in S$.  From Theorem \ref{ss1}, there exists an effective $\mathbb{Q}$-divisor $\Delta_{\mu}$ on $X_{\mu}$ such that $(X_{\mu},\Delta_{\mu})$ is lc and $K_{X_{\mu}}+\Delta_{\mu} \sim_{\mathbb{Q}} 0$.  We see also $$K_X \cdot C=K_{X_{\mu}} \cdot C_{\mu}= -\Delta_{\mu} \cdot C_{\mu} \leq 0$$ since $C_{\mu}$ is movable.  This is a contradiction to the assumption and shows that $-K_X$ is pseudo effective.

For the latter statement, it holds that $-K_{X_{\mu}}$ is big by Theorem \ref{ss1}.  Thus we see that $K_{X_\mu} \not \sim_{\mathbb{Q}}0$. If $K_X$ is pseudo-effective, then $K_X \equiv 0$ by Lemma \ref{triviality},
moreover, $K_X \sim_{\mathbb{Q}} 0$ from \cite[V, 4.9 Corollary]{N}, \cite[Theorem 4.2]{ambro}, \cite{kawamata-abunnuzero}, or \cite{ckp-num} since $X$ has only Kawamata log terminal singularities by  Proposition \ref{just singularities}.  Therefore, $K_{X_{\mu}}$ is also $\mathbb{Q}$-linearly trivial. This is a contradiction. Thus $K_X$  is not pseudo-effective. \end{proof}

\begin{cor}\label{uniruled} Let $X$ be a $\mathbb{Q}$-factorial normal projective variety of dense globally $F$-split type. Then $X$ is uniruled or  a canonical variety with $K_X \sim_{\mathbb{Q}}0$. \end{cor}

\begin{proof}Assume $X$ is not uniruled. First we reduce to the case that $X$ has only Kawamata log terminal singularities.    Indeed Proposition \ref{just singularities} implies that $X$ has only log canonical  and  take a dlt blow-up $f: Y \to X$ such that
$$f^*K_X=K_Y+\Gamma$$
 and $(Y, \Gamma)$ is a $\mathbb{Q}$-factorial divisorial log terminal pair as in Theorem \ref{dltblowup}.  If $K_Y$ is not pseudo-effective, $Y$ is birationally equivalent to a Mori fiber space by  \cite[Corollary 1.3.3]{bchm}, in particular $X$ is uniruled by \cite[IV.1.3 Proposition]{Kollar2} since klt Fano varieties are rationally connected by   \cite[Corollary 1.3]{HM} and \cite[Theorem 1]{Zhang}.  Thus we may assume that $K_Y$ is pseudo-effective.

 Lemma \ref{triviality} implies  $K_Y \equiv 0$ since $-K_Y$ is pseudo-effective by Lemma \ref{pseudo-effective}. Thus it holds that  $\Gamma =0$ since  $-(K_Y+\Gamma)$ is pseudo-effective by Lemma \ref{pseudo-effective}. Thus $X$ has only Kawamata log terminal singularities and $f$ is small by the condition (iii) in Theorem \ref{dltblowup} and \cite[IV, Proposition 2.41]{KM}. Next take a birational morphism $g: Y' \to X$ such that
$$g^*K_X=K_{Y'}+\Gamma'$$
 and $(Y, \Gamma)$ is a $\mathbb{Q}$-factorial  terminal pair by \cite[Corollary 1.4.3]{bchm}.  By the same argument as above, we see that $\Gamma'=0$, thus  $X$ has only canonical singularities with $K_X \equiv 0$. Thus $K_X \sim_{\mathbb{Q}} 0$ by \cite[Theorem 8.2]{kawamata-iitaka}, \cite[V, 4.9 Corollary]{N}, \cite[Theorem 4.2]{ambro}, \cite{kawamata-abunnuzero}, or \cite{ckp-num}.  This concludes the proof of Corollary \ref{uniruled}.\end{proof}

 We come to the proof of the main result of the section, Theorem \ref{0gfrtype-rc}.

\begin{thm}\label{rat-con}Let $X$ be a $\mathbb{Q}$-Gorenstein  normal projective variety of globally $F$-regular type. Then $X$ is rationally connected. \end{thm}

\begin{proof} We proceed by induction on dimension. Note that $X$ has only Kawamata log terminal singularities by Proposition \ref{just singularities}. Taking a small $\mathbb{Q}$-factorization (\cite[Corollary 1.4.3]{bchm}), we may assume that $X$ is $\mathbb{Q}$-factorial.

We know $K_X$ is not pseudo-effective from Lemma \ref{pseudo-effective}. Note that $X$ has log terminal singularities from Theorem \ref{just singularities} (2).  By \cite[Corollary 1.3.3]{bchm}, there exists a minimal model program which terminates to a Mori fiber space $$f:X' \to Y.$$ In particular, we know also $X'$ and $Y$ are of globally $F$-regular type (cf. Lemma \ref{glFness for MMP}). We see that $Y$ is also $\mathbb{Q}$-factorial by \cite[Corollary 3.18]{KM}. By the induction hypothesis, $Y$ is rationally connected. Thus applying \cite[Theorem 1.1]{ghs} for $f$ and by the rational connectedness of a general fiber of  $f$ (cf.  \cite[Corollary 1.3]{HM} and \cite[Theorem 1]{Zhang}), we see that $X$ is also rationally chain connected. For a klt projective variety, rational chain connectedness is rational connectedness by \cite[Corollary 1.5]{HM}. Thus $X$ is also rationally connected. \end{proof}

\end{document}